\documentclass[a4paper, 11pt]{article}
\usepackage{amsmath,amssymb,esint,amscd,xspace,fancyhdr,color,authblk,srcltx,fontenc,bbm}
\usepackage{hyperref}
\usepackage{cite}

\newtheorem{theorem}{Theorem}[section]

\newtheorem{definition}[theorem]{Definition}

\newtheorem{lemma}[theorem]{Lemma}

\newtheorem{remark}[theorem]{Remark}

\newenvironment{proof}[1][Proof]{\textbf{#1.} }{\hfill\rule{0.5em}{0.5em}}
{\catcode`\@=11\global\let\AddToReset=\@addtoreset
\AddToReset{equation}{section}

\AddToReset{theorem}{section}

\title{Calder\'on-Zygmund-type estimates for singular quasilinear elliptic obstacle problems with measure data}
\author{M.-P. Tran\footnote{Corresponding author.} \thanks{Applied Analysis Research Group, Faculty of Mathematics and Statistics, Ton Duc Thang University, Ho Chi Minh City, Vietnam; \texttt{tranminhphuong@tdtu.edu.vn}},  T.-N. Nguyen\thanks{Group of Analysis and Applied Mathematics, Department of Mathematics, Ho Chi Minh City University of Education, Ho Chi Minh City, Vietnam; \texttt{nhannt@hcmue.edu.vn}}, P.-N. Huynh\thanks{Nguyen Du High School, Ho Chi Minh City, Vietnam; \texttt{hpnguyen.thptnguyendu@hcm.edu.vn}}}

\date{\today}

\begin{document}
 
\maketitle
\begin{abstract}
We deal with a global Calder\'on-Zygmund type estimate for elliptic obstacle problems of $p$-Laplacian type with measure data. For this paper, we focus on the singular case of growth exponent, i.e. $1<p \le 2-\frac{1}{n}$. In addition, the emphasis of this paper is in obtaining the Lorentz bounds for the gradient of solutions with the use of fractional maximal operators.

\medskip

\medskip

\medskip 

\medskip


\medskip


\end{abstract}   
                  
\tableofcontents
\newpage
\section{Introduction}\label{sec:intro}

Our goal in this paper is to establish a Calder\'on-Zygmund type estimate for solutions to the elliptic obstacle problems with right-hand side measure. These problems are related to quasilinear elliptic equations with measure data:
\begin{align}\label{eqs}
-\mathrm{div}\mathbb{A}(\nabla u,x) = \mu, \quad \text{in} \ \Omega,
\end{align}
where $\Omega \subset \mathbb{R}^n$ is an open bounded domain with $n \ge 2$; $\mu$ is a bounded Radon measure on $\Omega$ that has finite total mass $|\mu|(\Omega)<+\infty$ (simply written $\mu \in \mathcal{M}_b(\Omega)$) and the quasi-linear operator $\mathbb{A}: \mathbb{R}^n \times \Omega \rightarrow \mathbb{R}^n$ is a vector field such that $\mathbb{A}(\eta,\cdot)$ is measurable in $\Omega$ for every $\eta \in \mathbb{R}^n$, $\mathbb{A}(\cdot,x)$ is continuous in $\mathbb{R}^n$ for almost every $x \in \Omega$. Due to the presence of an obstacle, here we are interested in problems with measure data related to~\eqref{eqs}, simply denoted by $\mathcal{P}(\mathbb{A},\mu,\psi)$, where the obstacle function $\psi \in W^{1,p}(\Omega)$ satisfies
\begin{align}\label{cond-psi}
\psi \le 0 \mbox{ on }  \partial\Omega \mbox{ and } \mathrm{div} \mathbb{A}(\nabla \psi,\cdot) \in L^1(\Omega).
\end{align}
The operator $\mathbb{A}$ is further assumed to satisfy both ellipticity and growth conditions: there exist constants $p \in (1,\infty)$, $\Upsilon>0$ such that 
\begin{align}\label{cond:A} 
\left| \mathbb{A}(\eta,x) \right|  \le \Upsilon |\eta|^{p-1} \ \mbox{ and } \ \langle \mathbb{A}(\eta_1,x)-\mathbb{A}(\eta_2,x), \eta_1 - \eta_2 \rangle  \ge \Upsilon^{-1} \Phi(\eta_1, \eta_2),
\end{align}
for almost every $x$ in $\Omega$ and every $\eta$, $\eta_1$, $\eta_2 \in \mathbb{R}^n \setminus \{0\}$, where the function $\Phi$ is defined by
\begin{align}\label{def:Phi}
\Phi(\eta_1, \eta_2) : = \left(|\eta_1|^2 + |\eta_2|^2 \right)^{\frac{p-2}{2}}|\eta_1 - \eta_2|^2, \quad \eta_1, \, \eta_2 \in \mathbb{R}^n.
\end{align}

It should be noted that the operator considered on left-hand side is more general and includes the $p$-Laplacian as a special case. A typical example of the problem~\eqref{eqs} is the $p$-Laplacian equation $-\Delta_p u = \mu$ when $\mathbb{A}$ is defined by $\mathbb{A}(\eta,x) = |\eta|^{p-2}\eta$. And in this study, we properly emphasize that the growth exponent $p$ is a number such that 
\begin{align}\label{range-p}
1<p \le 2-\frac{1}{n}.
\end{align}

Obstacle problems have been derived as models of many physical phenomena like porous media propagation, elasto-plasticity, tortion problems, financial mathematics, etc. To the best of our knowledge, solutions of the quasilinear elliptic obstacle problems can be naturally governed by variational inequalities. The theory of obstacle problems, that connected with variational inequalities and free boundary problems, has its origins in the calculus of variations when one seeks to maximize or minimize a functional. After decades of development, the study of obstacle problems has wide-ranging applications in various fields such as economics, biology, mechanics, computer science, engineering, etc and we refer the reader to ~\cite{KS2000, Friedman, Rod1987} for more applications, further details. Starting with some pioneering works by Stampachia and Lion in~\cite{LS1967}; Caffarelli in~\cite{Caffa1998}, there has been substantial amount of research pertaining to obstacle problems, and specifically on the existence and regularity theory for such problems. Along with the existence theory discussed in a number of contributions as~\cite{DL2002, BP2005, Leone}, there have been various results falling into the scope of regularity theory for obstacle problems.  For instance, we can consult~\cite{Choe1991} for the presentation of $C^{0,\alpha}$ and $C^{1,\alpha}$ estimates; H\"older continuity addressed in~\cite{EH2011}. Moreover, for the obstacle problems in divergence form, B\"ogelein \textit{et al.} in~\cite{BDM2011} established a local Calder\'on-Zygmund estimate for solutions to parabolic/ellipic variational inequalities. Later, Byun and his coworkers extended these results to the global ones up to the boundary. Otherwise, regularity results for obstacle problems with $p(x)$-growth have been developed by many authors in~\cite{EH2011, BCP2021} a few years ago. 

Before going into the details, let us briefly remind some known results related to  obstacle problem $\mathcal{P}(\mathbb{A},\mu,\psi)$, that have been extensively studied in the last years. As far as we know, obstacle problems with measure $\mu$ on the right-hand side have been considered in~\cite{BC1999, Leone2, OR2001, DL2002} along with the special attention paid to nonlinear equations with singular data ($L^1$-data or measures). With the presence of measure source term $\mu$ (being a bounded Radon measure) on the right-hand side, it makes the study of regularity theory more challenging.  It became an important subject investigated by a number of authors due to the notion of solutions and we address the reader to~\cite{BG1992, Maso1999, Min2007} along with references given therein for further reading. In general, obstacle problem with measure data $\mathcal{P}(\mathbb{A},\mu,\psi)$ cannot be described in an usual variational sense when right-hand side $\mu \in \mathcal{M}_b(\Omega)$ (due to the existence of solutions).  Therefore, solutions to such problem will be intepreted in a `unusual' weak sense.  With the aim of giving such a special distributional sense to solutions to $\mathcal{P}(\mathbb{A},\mu,\psi)$, Scheven in~\cite{Scheven2012} introduced a notion of so-called \emph{a limit of approximating solution} whose existence is reasonable due to the seminal works~\cite{Benilan1995, BC1999, BG1992} (see Definition~\ref{def:SOLA}). However, for a significant particular case, when $\mu \in W^{-1,p'}(\Omega)$ with $p' = \frac{p}{p-1}$, our obstacle problem $\mathcal{P}(\mathbb{A},\mu,\psi)$ will be governed by the following elliptic variational inequality problem:  Find $u \in \mathcal{S}_0$ such that  
\begin{align}\label{eq:main}
\int_{\Omega} \langle \mathbb{A}(\nabla u,x), \nabla \varphi - \nabla u \rangle dx \ge  \int_{\Omega} (\varphi - u)\mu dx,
\end{align}
for every $\varphi \in \mathcal{S}_0$, where 
\begin{align}\label{def:S0}
\mathcal{S}_0 = \left\{v \in \mathcal{T}_0^{1,1}(\Omega): \ v \ge \psi \mbox{ a.e. in }  \Omega\right\}.
\end{align}
Taking a closer look at the variational problem \eqref{eq:main}, we herein denote by $\mathcal{T}^{1,r}(\Omega)$ ($r \ge 1$) the function spaces that consists of all measurable functions $\varphi: \ \Omega \to \mathbb{R}$ such that $T_k(\varphi) \in W^{1,r}(\Omega)$ for all $k \ge 0$, where $T_k: \mathbb{R} \to \mathbb{R}$ is a truncation operator defined by
\begin{align*}
T_k(z) = \max\{-k,\min\{z,k\}\}, \quad z \in \mathbb{R}.
\end{align*}
We denote by $\mathcal{T}_0^{1,r}(\Omega)$ the subset of $\mathcal{T}^{1,r}(\Omega)$ containing of the functions $\varphi \in \mathcal{T}^{1,r}(\Omega)$ such that for every $k >0$, there is a sequence $(\varphi_j^k)_{j \in \mathbb{N}} \subset C_0^{\infty}(\Omega)$ satisfying 
$$\varphi_j^k \to T_k(\varphi) \mbox{ in } L^1_{\mathrm{loc}}(\Omega) \mbox{ and } \nabla \varphi_j^k \to \nabla T_k(\varphi) \mbox{ in } L^r(\Omega) \mbox{ as } j \to \infty.$$

Besides, the bound of exponent $p$ also plays a crucial role in the proofs of gradient estimates even without obstacles and one can go through the works in~\cite{55QH4, MPT2018, HP2021, Dong2021} for detailed results and discussions in this direction. So far, Scheven in his celebrated papers~\cite{Scheven, Scheven2012} has derived the gradient and point-wise estimates for solutions to elliptic obstacle problems involving measure data via Wolff potentials, with the growth exponent $p>2-\frac{1}{n}$. Recently, Byun and his collaborators in~\cite{BCP2021} have also investigated the global gradient estimates for double obstacle problems with measure data when the variable exponential growth such that $p(\cdot) >2 -\frac{1}{n}$. As far as the measure datum $\mu$ is concerned, our intention in this paper is to deal with problem $\mathcal{P}(\mathbb{A},\mu,\psi)$ and the singular case \eqref{range-p} will be discussed here. It is known that even without any obstacle functions, one cannot expect the solution to belong to the classical Sobolev space $W^{1,1}$. For example, one of the most typical problems is  $-\Delta_p u = \delta_0$ in $\mathbb{R}^n$, where $\delta_0$ is the Dirac mass and in that case, the solution $u$ is given by $u(x) = C|x|^{-\frac{n-p}{p-1}}$, $C$ is a suitable constant. It can be seen that $u \in W^{1,1}_{\mathrm{loc}}(\mathbb{R}^n)$ if and only if $p> 2 - \frac{1}{n}$. For this reason, it requires an understanding of gradient in a different notion instead of the usual distribution sense. In~\cite{Benilan1995}, authors proved that if $u \in \mathcal{T}_{\mathrm{loc}}^{1,1}(\Omega)$ then there exists a measurable function $v: \, \Omega \to \mathbb{R}^n$ such that $\nabla T_k(u) = v \chi_{\{|u|\le k\}}$ a.e. in $\Omega$, for every $k \in \mathbb{N}$. As such, `gradient' of $u$ could be explained as $v$ and written `$\nabla u:=v$'. Toward the existence of solutions, authors in~\cite{BC1999, Leone2, OR2001} proposed a notion of solutions to $\mathcal{P}(\mathbb{A},\mu,\psi)$, known as \emph{a limit of approximating solution}. And then, the point-wise regularity of solutions in this sense to problems of type $\mathcal{P}(\mathbb{A},\mu,\psi)$ was also well understood through the works done by Scheven, as mentioned above.
 
Main results of this paper, regarding to the Calder\'on-Zygmund-type estimates, can be established in terms of the fractional maximal operators of gradient of both solutions and data in Lorentz spaces. More precisely, we shall infer that
\begin{align*}
\mathbf{M}_{\beta}(\mu) + \mathbf{M}_{\beta}\left(\mathrm{div} \mathbb{A}(\nabla \psi,\cdot)\right) \in L^{\frac{\gamma q}{p-1},\frac{\gamma s}{p-1}}(\Omega) \ \ \text{implies} \ \ \mathbf{M}_{\alpha}(|\nabla u|^{\gamma}) \in L^{q,s}(\Omega),
\end{align*}
for every $q,s$ and appropriate parameters $\alpha, \beta, \gamma$. In other words, it yields the following Lorentz-norm bounds for gradient of solutions
\begin{align*}
\|\mathbf{M}_{\alpha}(|\nabla u|^{\gamma})\|_{L^{q,s}(\Omega)} & \le  C \left\|\left[\mathbf{M}_{\beta}(\mu) + \mathbf{M}_{\beta}\left(\mathrm{div}\mathbb{A}(\nabla \psi,\cdot)\right)\right]^{\frac{\gamma}{p-1}}\right\|_{L^{q,s}(\Omega)}.
\end{align*}
Our approach here is inspired by recent works dealt with problems \eqref{eqs} without the fractional maximal functions, that lead to
\begin{align*}
 \left[\mathbf{M}_{1}(\mu)\right]^{\frac{1}{p-1}} \in L^{q,s}(\Omega) \ \ \text{implies} \ \ \nabla u \in L^{q,s}(\Omega).
\end{align*}
Reader may consult~\cite{Min2011, 55QH4, MPT2018} for non-obstacle problems and~\cite{BCP2021} for obstacle ones.

We note that the technique applied in the proofs mainly relies on the comparison estimates and the use of Calder\'on-Zygmund type covering arguments. The idea of this approach goes back to  Mingione \textit{et al.} in~\cite{AM2007, Min2011} and later it was improved, modified and developed in a lot of works treating the regularity results. Here, the key feature in our proofs is that we take advantage of the comparison procedures and fractional maximal functions in order to obtain desired results. The proof of our results goes through several steps. First step we construct the corresponding homogeneous problems and establish comparison estimates between the unique solution to those problems with solution $u$ to the obstacle $\mathcal{P}(\mathbb{A},\mu,\psi)$. Particularly, comparison scheme in step 1 is divided into two stages: we compare the solution $u$ with the solution of elliptic obstacle problems with frozen coefficients; and then with the unique solution to a homogeneous elliptic equation (cf. Lemma~\ref{lem:comp1} and Lemma~\ref{lem:comp2}). In the second step, these comparison results will be used to derive the level-set decay estimates for  solutions to original obstacle problem, that carried out in Theorem~\ref{theo:main-A}. The next step allows us to employ Calder\'on-Zygmund type covering argument of the level sets and certain properties of fractional maximal operators to establish desired estimates for solutions to obstacle problems (cf. Theorem~\ref{theo:main-B}). It emphasizes that gradient bounds of solutions here will be preserved under the fractional maximal operators and further, Calder\'on-Zygmund type estimates will be obtained in the setting of Lorentz spaces. It is also remarkable that in order to achieve the global regularity estimates for solutions to $\mathcal{P}(\mathbb{A},\mu,\psi)$, we have to impose an additional structural assumption on boundary $\partial\Omega$. In the present paper, we assume that boundary $\partial\Omega$ is flat in the sense of Reifenberg (we refer the reader to the next section for a precise definition) and moreover, the nonlinearity $\mathbb{A}$ satisfies a small bounded mean oscillation (BMO) with respect to the spatial variable. 

The main difficulty in our proofs is how to deal with the first step: establish the comparison results when the original obstacle problems involving measure datum $\mu$, for singular growth exponent $p \in \left(1,2-\frac{1}{n}\right]$. To this aim, let us describe the key idea underlying the main results here. In spirit to the ideas of earlier technique proposed by Benilan \textit{et al.} in \cite[Lemma 4.1]{Benilan1995}, where the authors showed that:  if $\omega \in \mathcal{T}_0^{1,p}(\Omega)$ such that 
$$\int_{\{|\omega|\le k\}} |\nabla \omega|^p dx \le k \Pi, \quad \mbox{for every } k>0,$$
then
$$\|\omega\|_{L^{\frac{n(p-1)}{n-p},\infty}(\Omega)} \le C \Pi^{\frac{1}{p-1}},$$
for a constant $\Pi>0$. In this study, by taking advantage of this idea of \cite[Lemmas 4.1 and 4.2]{Benilan1995}, we derive the general comparison estimates of the following type:  if $u \in \mathcal{T}^{1,p}(\Omega)$ and $v \in u + \mathcal{T}^{1,p}_0(\Omega)$ such that $|\nabla u| \in L^{2-p}(\Omega)$ and
\begin{align*}
\int_{B \cap \{h < |u-v| < k+h\}} \Phi(\nabla u, \nabla v) dx \le k\Pi, \quad \mbox{for every } k, h>0,
\end{align*}
then it holds that
\begin{align*}
\|u-v\|_{L^{\tilde{p},\infty}(B)}  \le C  \Pi^{\frac{1}{p-1}} + C \Pi \int_{B} |\nabla u|^{2-p}dx,
\end{align*}
for any unit ball $B$ in $\Omega$ and $\tilde{p}$ will be precisely given in \eqref{p-tilde}. In our proof strategy, we are particularly interested in the case $1 <p\le 2-\frac{1}{n}$, which can not ensure that 
$$\int_{B \cap \{|u-v| \le k\}} \Phi(\nabla u, \nabla v) dx \le k\Pi$$ 
holds for every $k>0$. Moreover, it is more difficult to handle when the integral over an open set $\{h < |u-v| < k+h\}$ that does not contain the origin (instead of the set $\{|u-v|<k\}$ as in~\cite{Benilan1995}). The idea to prove it comes from recent works~\cite{55QH4, Dong2021} for quasilinear elliptic equations. As far as we know, there is less work on regularity theory for solutions to quasilinear elliptic measure data problems dealt with the growth exponent $1<p \le 2-\frac{1}{n}$, even without obstacles. Therefore, this paper is a contribution to the study of regularity theory for obstacle problems, especially when right-hand side is a measure. In particular, a global gradient estimate for solutions will be established in this paper in terms of fractional maximal functions and in the setting of Lorentz spaces. Moreover, it should be worth noting that the results of this paper can be extended to the class of parabolic problems.

The rest part of this work is arranged as follows. In the next section we present some basic notation, definitions and some imposed assumptions on which our problems rely. Section~\ref{sec:def} is also devoted to the statements of main results in this paper via two important theorems. Section~\ref{sec:pre} focuses on providing some preliminary comparisons results via main lemmas, that play a key role in the rest of the paper. With these preparatory lemmas in hand, Section~\ref{sec:levelset} allows us to prove the level-set inequality in Theorem~\ref{theo:main-A}, a useful tool when dealing with main results. Finally, in Section~\ref{sec:main}, we end up with the proofs of main theorems, where the level-set decay estimates and the global gradient estimates  will be proceeded.

\section{Definitions and statement of results}\label{sec:def}

\textbf{Notation.} In what follows, $\Omega \subset \mathbb{R}^n$ will denote an open bounded domain, for $n \ge 2$. For simplicity of notation, we shall employ $B_\rho(x_0)$ in place of the open ball with center $x_0 \in \mathbb{R}^n$ and radius $\rho>0$; denote by 
\begin{align*}
\fint_{D} \varphi(x) dx = \frac{1}{\mathcal{L}^n(D)} \int_{D} \varphi(x) dx
\end{align*}
the integral average of $\varphi$ over the set $D$. We further use $\mathcal{L}^n(D)$ to stand for the Lebesgue measure of $D$ in $\mathbb{R}^n$, and $\text{diam}(D)$ for the diameter of a set $D \subset \mathbb{R}^n$, i.e.
\begin{align*}
\text{diam}(D) = \sup_{x, y \in D} |x-y|.
\end{align*}
Moreover, for any measurable function $\varphi$ in $\Omega$ we will denote by $\{|\varphi|>\lambda\}$ the level set $\{x \in \Omega: \, |\varphi(x)|>\lambda\}$. 
For the brevity of notation, we state here that the letter $C$ is used to represent a technical constant depending only on some prescribed quantities,  $C$ is always large than one and its value may change at different occurrences throughout the paper. It remarks that instead of repeating in every statement, let us label $\texttt{data}$ for the set of parameters that problems depend only on, i.e.
\begin{align*}
\texttt{data} = \left(n,p,\Upsilon,\mathrm{diam}(\Omega)\right),
\end{align*}
and we shall adopt $C(\texttt{data})$ to illustrate the relevant dependencies of $C$ on parameters in $\texttt{data}$. For reasons of readability, we shall regard 
\begin{align}\label{def-chi-1}
\chi_1 & = \chi_{\left\{p>\frac{3n-2}{2n-1}\right\}} = \begin{cases} 1, \mbox{ if } p> \frac{3n-2}{2n-1}, \\ 0, \mbox{ if } p \le \frac{3n-2}{2n-1}, \end{cases} \\ \label{def-chi-2} 
\chi_2 & = 1 - \chi_1 = \chi_{\left\{p\le \frac{3n-2}{2n-1}\right\}}.
\end{align}

In the present work, in order to obtain global Lorentz regularity estimates, we impose two additional assumptions on our problems of type $\mathcal{P}(\mathbb{A},\mu,\psi)$. Here, domain assumption specifies $\Omega$ has its boundary being sufficiently flat in the sense of Reifenberg and moreover, the $(r_0,\delta)$-BMO condition exploited on the $\mathbb{A}$. It is worth mentioning that these two assumptions are minimal regularity requirements on the boundary $\partial\Omega$ and nonlinearity $\mathbb{A}$ to achieve some technical results of our  problems. For simplicity of notation, if $\mathbb{A}$ satisfies \eqref{cond:BMO} and $\Omega$ is a $(r_0,\delta)$-Reifenberg flat domain with positive numbers $r_0$ and $\delta$, then we write $(\mathbb{A};\Omega) \in \mathcal{H}_{r_0,\delta}$.  Useful definitions are given as follows.

\begin{definition}\label{def:Reifenberg}
Let $0 < \delta < 1$ and $r_0>0$, if for each $z \in \partial \Omega$ and $\rho \in (0,r_0]$, there is a coordinate system $\{x_1',x_2',...,x_n'\}$ with origin at $z$ satisfying
\begin{align*}
B_{\rho}(z) \cap \{x_n' > \delta \rho\} \subset B_{\rho}(z) \cap \Omega \subset B_{\rho}(z) \cap \{x_n' > -\delta \rho\},
\end{align*}
then we will call that $\Omega$ is a $(r_0,\delta)$-Reifenberg flat domain. Here we use $\{x_n' > c\}$ instead of the set $\{(x_1', x_2', ..., x_n'): \ x_n' > c\}$.
\end{definition}

It remarks that for given regularity parameter $\delta$, the class of domains satisfying Reifenberg flatness condition contains all Lipschitz domains with small Lipschitz constants or even domains with fractal boundaries. The detail discussions can be found in~\cite{BW2004} and references therein.
 
\begin{definition}\label{def:smallBMO}
The operator $\mathbb{A}$ is called to satisfy a $(r_0,\delta)$-BMO condition if
\begin{align}\label{cond:BMO}
[\mathbb{A}]^{r_0} = \sup_{\xi \in \mathbb{R}^n, \, 0<\rho\le r_0} \fint_{B_{\rho}(\xi)} \left(\sup_{\eta \in \mathbb{R}^n \setminus \{0\}} \frac{|\mathbb{A}(\eta,x) - \overline{\mathbb{A}}_{B_{\rho}(\xi)}(\eta)|}{|\eta|^{p-1}}\right) dx \le \delta,
\end{align}
where $\overline{\mathbb{A}}_{B_{\rho}(\xi)}(\eta)$ is the average of $\mathbb{A}(\eta,\cdot)$ over $B_{\rho}(\xi)$.
\end{definition}

Next, for the convenience of the reader, we also respectively include here the definitions of Lorentz spaces and fractional maximal operators $\mathbf{M}_\alpha$, on which our main results focus.

\begin{definition}\label{def:Lorentz} 
Let $q \in (0,\infty)$ and $0 <s \le \infty$, the Lorentz space $L^{q,s}(\Omega)$ is defined by
\begin{align*}
L^{q,s}(\Omega) = \left\{\varphi: \ \varphi \mbox{ is measurable on } \Omega \mbox{ satisfying } \|\varphi\|_{L^{q,s}(\Omega)} < \infty\right\},
\end{align*}
where the quasi-norm $\|\cdot\|_{L^{q,s}(\Omega)}$ is given by
\begin{align*}
\|\varphi\|_{L^{q,s}(\Omega)} := \begin{cases} \displaystyle{\left[ q \int_0^\infty{ \lambda^{s-1}\mathcal{L}^n \left( \{|\varphi|>\lambda\} \right)^{\frac{s}{q}} d\lambda} \right]^{\frac{1}{s}}} & \quad \mbox{ if } s<\infty,\\  
\displaystyle{\sup_{\lambda>0}{\lambda \left[\mathcal{L}^n\left(\{|\varphi|>\lambda\}\right)\right]^{\frac{1}{q}}}} & \quad \mbox{ if } s=\infty.\end{cases}
\end{align*}
\end{definition}

\begin{definition}\label{def:Malpha}
The fractional maximal operator $\mathbf{M}_{\alpha}$ for $\alpha \in [0, n]$ is defined by:
\begin{align} \nonumber 
\mathbf{M}_\alpha \varphi(z) = \sup_{\rho>0}{{\rho}^{\alpha} \fint_{B_{\rho}(z)}{|\varphi(y)|dy}}, \quad z \in \mathbb{R}^n, \ \varphi \in L^1_{\mathrm{loc}}(\mathbb{R}^n).
\end{align}
The operator $\mathbf{M}_{0}$ is the Hardy-Littlewood operator $\mathbf{M}$ given by:
\begin{align}\nonumber 
\mathbf{M}\varphi(z) = \sup_{\rho >0}{\fint_{B_{\rho}(z)}|\varphi(y)|dy}, \quad z \in \mathbb{R}^n, \ \varphi \in L^1_{\mathrm{loc}}(\mathbb{R}^n).
\end{align}
\end{definition}
We now recall the bounded property of $\mathbf{M}_{\alpha}$. A detail proof can be seen in \cite[Lemma 3.3]{PNJFA}.
\begin{lemma}\label{lem:M_alpha}
Let $s \ge 1$ and $\alpha \in \left[0,\frac{n}{s}\right)$, there holds
\begin{align*}
\|\mathbf{M}_{\alpha} \varphi\|_{L^{\frac{n s}{n - \alpha s},\infty}(\mathbb{R}^n)}  \le C \|\varphi\|_{L^{s}(\mathbb{R}^n)},
\end{align*}
for all $\varphi \in L^s(\mathbb{R}^n)$, where $C=C(n,s,\alpha)>0$. 
\end{lemma}

To formulate our main results, it is important to give a notion of solutions to $\mathcal{P}(\mathbb{A},\mu,\psi)$. We briefly recall here the \emph{limit of approximating solutions} of such obstacle problems that introduced by Scheven in~\cite{Scheven}. 

\begin{definition}\label{def:SOLA}
We say that $u \in \mathcal{S}_0$ is a limit of approximating solutions of the problem $\mathcal{P}(\mathbb{A},\mu,\psi)$ if there are functions
\begin{align*}
\mu_k \in L^1(\Omega) \cap W^{-1,p'}(\Omega) \ \mbox{with} \ \mu_k \to \mu 
\end{align*}
in the narrow topology of measures in $\mathcal{M}_b(\Omega)$ and solutions $u_k \in W^{1,p}(\Omega) \cap \mathcal{S}_0$ of the variational formula
\begin{align}\label{eq:var-form}
\int_{\Omega} \langle \mathbb{A}(\nabla u_k,x), \nabla \varphi - \nabla u_k \rangle dx \ge  \int_{\Omega} (\varphi - u_k)\mu_k dx,
\end{align}
for all $\varphi \in u_k + W^{1,p}_0(\Omega) \cap \mathcal{S}_0$, such that
\begin{align}\label{conver}
\begin{cases} u_k \to u \mbox{ a.e. on } \Omega, \\ u_k \to u \mbox{ in } L^r(\Omega) \mbox{ for every } r \in \left(0,\frac{n(p-1)}{n-p}\right),\\ \nabla u_k \to \nabla u \mbox{ in } L^{s}(\Omega) \mbox{ for every } s \in \left(0,\frac{n(p-1)}{n-1}\right).  \end{cases}
\end{align}
\end{definition}

It is to be noticed that the existence of solutions in the sense of Definition~\ref{def:SOLA} follows the study of energy solutions discussed in earlier works~\cite{Benilan1995, BC1999, BG1992}.

\textbf{Main results.} We are now ready to state our main results of this paper. In this regard, Theorem~\ref{theo:main-A} captures the level-set inequality involving gradient of solutions to $\mathcal{P}(\mathbb{A},\mu,\psi)$, with the growth exponent $p \in \left(1,2-\frac{1}{n}\right]$. In a related context, it is worth highlighting that the idea of level-set decay estimates with $\mathbf{M}_\alpha$ can be understood in the sense of \emph{fractional maximal distribution functions} and interested reader can go through our previous work in~\cite{PNJFA}.

\begin{theorem}\label{theo:main-A}
Let $1<p\le 2-\frac{1}{n}$ and $\psi \in W^{1,p}(\Omega)$ satisfying~\eqref{cond-psi}. Suppose that $u$ is a solution to obstacle problem $\mathcal{P}(\mathbb{A},\mu,\psi)$ such that $|\nabla u| \in L^{2-p}(\Omega)$ with given data $\mu \in \mathcal{M}_b(\Omega)$. For every $\gamma \in (\gamma_{1},\gamma_{2})$ with
\begin{align}\label{p-star}
\gamma_{1}  & = \chi_1 (2-p)  \mbox{ and } \gamma_{2}  = \min\left\{\frac{np}{3n-2}; \frac{(p-1)n}{n-1}\right\}, 
\end{align}
and $\alpha, \beta, \sigma \in [0,n)$ satisfying 
\begin{align}\label{beta-sigma}
\beta = 1 + \frac{(p-1)\alpha}{\gamma} \ \mbox{ and } \ \sigma = \frac{(2-p)\alpha}{\gamma},
\end{align} 
one can find some constants $a>0$, $\varepsilon_0 \in (0,1)$ and $\delta \in (0,\frac{1}{2})$ such that if $(\mathbb{A};\Omega) \in \mathcal{H}_{r_0,\delta}$ for some $r_0>0$ then
\begin{align}\label{dist-ineq}
& \mathcal{L}^n \left(\left\{\mathbf{M}_{\alpha}(|\nabla u|^{\gamma})> a\lambda\right\}\right) \le \chi_2 \mathcal{L}^n \left(\left\{\left[\mathbf{M}_{\sigma}(|\nabla u|^{2-p})\right]^{\frac{\gamma}{2-p}} > \varepsilon^{-\gamma}\lambda \right\}\right) \notag \\
& \hspace{3cm} + \mathcal{L}^n \left(\left\{\left[\mathbf{M}_{\beta}(\mu) + \mathbf{M}_{\beta}\left(\mathrm{div}\left(\mathbb{A}(\nabla \psi,\cdot)\right)\right)\right]^{\frac{\gamma}{p-1}} >  \varepsilon^2 \lambda\right\}\right) \notag \\
& \hspace{5cm} + C \varepsilon \mathcal{L}^n \left(\left\{\mathbf{M}_{\alpha}(|\nabla u|^{\gamma})> \lambda\right\}\right),
\end{align}
for any $\lambda>0$ and $\varepsilon \in (0,\varepsilon_0)$, where $\chi_1, \chi_2$ are given as in~\eqref{def-chi-1}-\eqref{def-chi-2}.
Here the constants $a$, $\varepsilon_0$, $\delta$ and $C$ depend on $\tilde{\mathtt{data}} = (\gamma,\alpha,\mathtt{data})$.
\end{theorem}

As a consequence of Theorem~\ref{theo:main-A} and the use of Calder\'on-Zygmund type covering argument, the next result specifies global Lorentz bounds for gradient of solutions to $\mathcal{P}(\mathbb{A},\mu,\psi)$ via fractional maximal functions, in the statement of Theorem~\ref{theo:main-B} as follows. 

\begin{theorem}\label{theo:main-B}
Under hypotheses of Theorem~\ref{theo:main-A}, let $0<q<\infty$ and $0<s\le\infty$. There exist $\tilde{\epsilon}>0$ and $\delta>0$ depending $q,s,\tilde{\mathtt{data}}$ such that if $(\mathbb{A};\Omega) \in \mathcal{H}_{r_0,\delta}$ for some $r_0>0$ then
\begin{align}\label{est:theo-B}
\|\mathbf{M}_{\alpha}(|\nabla u|^{\gamma})\|_{L^{q,s}(\Omega)} & \le \epsilon \chi_2 \left\|\left[\mathbf{M}_{\sigma}(|\nabla u|^{2-p})\right]^{\frac{\gamma}{2-p}}\right\|_{L^{q,s}(\Omega)} \notag \\
&\qquad + C(\epsilon) \left\|\left[\mathbf{M}_{\beta}(\mu) + \mathbf{M}_{\beta}\left(\mathrm{div}\left(\mathbb{A}(\nabla \psi,\cdot)\right)\right)\right]^{\frac{\gamma}{p-1}}\right\|_{L^{q,s}(\Omega)},
\end{align}
for every $\epsilon \in (0,\tilde{\epsilon})$, where $C = C(\epsilon,q,s,\tilde{\mathtt{data}})>0$.
\end{theorem}
\textbf{Some discussions.} We also include here some discussions which connect our main results.

\begin{remark}
As most of quasilinear elliptic problems naturally are modeled from the $p$-Laplacian, our results here can cover the basic problems involving $p$-Laplacian type operator, when 
$$\mathbb{A}(\eta,x) = |\eta|^{p-2}\eta, \quad (\eta,x) \in \mathbb{R}^n \times \Omega.$$
\end{remark}

\begin{remark}\label{rmk:1}
Let us now be a bit more precise and explain about the assumption of $\gamma$ in Theorem~\ref{theo:main-A}. As aforementioned, solution $u$ to $\mathcal{P}(\mathbb{A},\mu,\psi)$ is studied in the sense of \emph{limit of approximating solutions} as in Definition~\ref{def:SOLA}, and due to \eqref{conver}, it is important to ensure that the intersection between two ranges $(\gamma_1,\gamma_2)$ and $\left(0,\frac{n(p-1)}{n-1}\right)$ is non-empty. Indeed, from~\eqref{p-star} one can see that
\begin{align}\notag
& \gamma_{1} = \begin{cases} 0, & \mbox{if } \  1< p \le \frac{3n-2}{2n-1},\\ 2-p, &  \mbox{if } \ \frac{3n-2}{2n-1} < p \le 2-\frac{1}{n}, \end{cases} \\ \notag
& \gamma_{2} = \begin{cases} \frac{n(p-1)}{n-1}, &  \mbox{if } \ 1< p \le \frac{3n-2}{2n-1},\\ \frac{np}{3n-2}, &  \mbox{if } \ \frac{3n-2}{2n-1} < p \le 2-\frac{1}{n}. \end{cases}  
\end{align}
Thus, if $1<p\le \frac{3n-2}{2n-1}$ then 
\begin{align}\label{est:rmk-1}
0=\gamma_{1}<\gamma_{2}=\frac{n(p-1)}{n-1} \le \frac{np}{3n-2} \le 2-p<1,
\end{align}
and otherwise if $\frac{3n-2}{2n-1}<p \le 2 - \frac{1}{n}$ then
\begin{align}\label{est:rmk-2}
0<\gamma_{1}=2-p <  \frac{np}{3n-2} = \gamma_{2} < \frac{n(p-1)}{n-1} \le 1.
\end{align} 
\end{remark}

\begin{remark}
If $\frac{3n-2}{2n-1}<p\le 2 - \frac{1}{n}$ then $\chi_2 = 0$. In this case the inequality~\eqref{est:theo-B} reduces to
\begin{align*}
\|\mathbf{M}_{\alpha}(|\nabla u|^{\gamma})\|_{L^{q,s}(\Omega)} & \le C \left\|\left[\mathbf{M}_{\beta}(\mu) + \mathbf{M}_{\beta}\left(\mathrm{div}\left(\mathbb{A}(\nabla \psi,\cdot)\right)\right)\right]^{\frac{\gamma}{p-1}}\right\|_{L^{q,s}(\Omega)}.
\end{align*}
Unfortunately, in another case when $\chi_2 = 1$, we may not obtain the above estimate. Indeed, since $0<\gamma<\gamma_{2}\le 2-p$ from~\eqref{est:rmk-1}, by H\"older's inequality one has
\begin{align*}
\mathbf{M}_{\alpha}(|\nabla u|^{\gamma})(x) & = \sup_{r>0} \left(r^{\alpha}\fint_{B_r(x)}|\nabla u|^{\gamma}dz\right) \\
& \le  \sup_{r>0} \left(r^{\frac{\alpha(2-p)}{\gamma}}\fint_{B_r(x)}|\nabla u|^{2-p}dz\right)^{\frac{\gamma}{2-p}} \\
& = \left[\mathbf{M}_{\sigma}(|\nabla u|^{2-p})(x)\right]^{\frac{\gamma}{2-p}}, \ \mbox{ for any } x \in \mathbb{R}^n.
\end{align*}
However, a nice feature here is that the coefficient of $\mathbf{M}_{\sigma}(|\nabla u|^{2-p})$ on the right-hand side of \eqref{est:theo-B} are not affected due to a near zero $\epsilon$. 
\end{remark}

\begin{remark}
A special case when $\alpha=0$ and $q>2-p$, one can use the boundedness property of $\mathbf{M}$ on $L^{\frac{q}{2-p},s}(\Omega)$ to imply that 
\begin{align*}
\left\|\left[\mathbf{M}(|\nabla u|^{2-p})\right]^{\frac{\gamma}{2-p}}\right\|_{L^{q,s}(\Omega)} \le C \left\||\nabla u|^{\gamma}\right\|_{L^{q,s}(\Omega)}.
\end{align*}
Furthermore, it follows from~\eqref{est:theo-B} that
\begin{align*}
\|\nabla u\|_{L^{q\gamma,s\gamma}(\Omega)} & \le C \left\|\left[\mathbf{M}_{1}(\mu) + \mathbf{M}_{1}\left(\mathrm{div} \mathbb{A}(\nabla \psi,\cdot)\right)\right]^{\frac{1}{p-1}}\right\|_{L^{q\gamma,s\gamma}(\Omega)}.
\end{align*}
\end{remark}

\begin{remark}
Result in Theorem~\ref{theo:main-B} can be extended to the weighted Lorentz spaces $L^{q,s}_{\omega}(\Omega)$, for a given Muckenhoupt weighted $\omega \in \mathbf{A}_{\frac{q}{2-p}}$ with $q>2-p$. To do this, we will establish a level set decay inequality which is similar to~\eqref{dist-ineq} as below
\begin{align*}
& \omega \left(\left\{\mathbf{M}_{\alpha}(|\nabla u|^{\gamma})> a\lambda\right\}\right) \le \chi_2 \omega \left(\left\{\left[\mathbf{M}_{\sigma}(|\nabla u|^{2-p})\right]^{\frac{\gamma}{2-p}} > \varepsilon^{-\gamma}\lambda \right\}\right) \notag \\
& \hspace{3cm} + \omega \left(\left\{\left[\mathbf{M}_{\beta}(\mu) + \mathbf{M}_{\beta}\left(\mathrm{div}\left(\mathbb{A}(\nabla \psi,\cdot)\right)\right)\right]^{\frac{\gamma}{p-1}} >  \varepsilon^2 \lambda\right\}\right) \notag \\
& \hspace{5cm} + C \varepsilon \omega \left(\left\{\mathbf{M}_{\alpha}(|\nabla u|^{\gamma})> \lambda\right\}\right). 
\end{align*}
Here, we use $\omega(D) = \int_D \omega(x)dx$ for simplicity. In the sequence, we also  recall the Muckenhoupt class $\mathbf{A}_t$ for $t>1$ defined by
\begin{align*}
\mathbf{A}_t = \left\{\omega \in L^1_{\mathrm{loc}}(\mathbb{R}^n;\mathbb{R}^+): \ [\omega]_{\mathbf{A}_t} < \infty\right\},
\end{align*}
where 
\begin{align*}
[\omega]_{\mathbf{A}_t} = \sup_{\rho>0, \, z \in \mathbb{R}^n} \left(\fint_{B_{\rho}(z)}\omega(x)dx\right) \left(\fint_{B_{\rho}(z)}\omega(x)^{-\frac{1}{t-1}}dx\right)^{t-1}.
\end{align*}
\end{remark}

\begin{remark}
Results can be generalized to the quasilinear parabolic obstacle problems. For instance, one considers the obstacle problems that related to following quasilinear parabolic equations of the type
\begin{align*}
\begin{cases}
u_t - \mathrm{div}\mathbb{A}(\nabla u,x) &= \mu, \qquad \mathrm{in} \ \ \Omega_T = \Omega \times (0,T), \\
u &= 0, \qquad \mathrm{on} \ \ \partial \Omega_T,\\
u(\cdot, 0) &= u_0, \quad \ \, \mathrm{in} \ \ \Omega.
\end{cases}
\end{align*}
And we are interested in solution $u$ belonging to the set
\begin{align*}
\mathcal{K}(\Omega_T) = \{v \in C^0([0,T];W_0^{1,p}(\Omega)) \cap C([0,T];L^2(\Omega)): v \ge \psi \  \text{a.e. in} \ \Omega_T\},
\end{align*}
for a given obstacle function $\psi: \Omega \times [0,T] \to \mathbb{R}$ satisfying
\begin{align*}
\psi \in L^p(0,T; W^{1,p}(\Omega)) \cap C([0,T];L^2(\Omega)), \ \psi_t \in L^{p'}(\Omega_T); \ \psi \le 0 \ \  \text{on}\ \ \partial\Omega_T.
\end{align*}
The parabolic problems with measure data have been attracting more and more attention from researchers in recent years. We believe that our technique and results in this paper would also be established for a class of parabolic obstacle problems when $p>1$.

\end{remark}

\begin{remark}
In discussion, we also expect that these results could be extended into the research on quasilinear obstacle problems with $p(x)$-growth, where $1<p^- \le p(\cdot) \le p^+ \le 2-\frac{1}{n}$.
\end{remark}

\section{Preliminary comparison results}
\label{sec:pre}

In this section, we are devoted to some preliminary lemmas that are important to prove main results. Section is divided into two parts, where the first part is mainly concerned with some abstract comparison results between operators and second one includes a series of comparison results between solutions to $\mathcal{P}(\mathbb{A},\mu,\psi)$ and corresponding problems (homogeneous obstacle problems with frozen coefficients, homogeneous elliptic equations).
 
\subsection{Abstract results}
\begin{lemma}\label{lem:Tech2}
Let $p \in (1,2)$, $s>0$ and $\gamma \in (0,ps)$. Assume that $B$ is a set of $\mathbb{R}^n$ and $v_1, v_2 \in L^{\gamma}(B;\mathbb{R}^n)$ such that $\Phi(v_1,v_2) \in L^{s,\infty}(B)$. For each $\varepsilon \in (0,1)$, there exists $C = C(p,s,\gamma) \varepsilon^{1 - \frac{2}{p}}>0$ such that
\begin{align}\label{est:Tech2}
\int_{B} |v_1 - v_2|^{\gamma} dx \le \varepsilon \int_{B} |v_1|^{\gamma} dx + C [\mathcal{L}^n(B)]^{1-\frac{\gamma}{ps}} \|\Phi(v_1,v_2)\|_{L^{s,\infty}(B)}^{\frac{\gamma}{p}},
\end{align}
where the function $\Phi$ is given as in~\eqref{def:Phi}.
\end{lemma}
\begin{proof}
According to the definition of $\Phi$ in~\eqref{def:Phi}, it allows us to get
\begin{align}\nonumber
|v_1 - v_2|^{\gamma} & =   [\Phi(v_1,v_2)]^{\frac{\gamma}{2}} \left(|v_1|^2 + |v_2|^2 \right)^{\frac{(2-p)\gamma}{4}} \\ \nonumber
& \le 2^{\frac{(2-p)\gamma}{4}}  [\Phi(v_1,v_2)]^{\frac{\gamma}{2}} \left(|v_1|^2 + |v_1-v_2|^2 \right)^{\frac{(2-p)\gamma}{4}} \\ \label{est:Y-I}
& \le 2^{\frac{(2-p)\gamma}{2}}  [\Phi(v_1,v_2)]^{\frac{\gamma}{2}} |v_1|^{\frac{(2-p)\gamma}{2}} + 2^{\frac{(2-p)\gamma}{2}}  [\Phi(v_1,v_2)]^{\frac{\gamma}{2}} |v_1-v_2|^{\frac{(2-p)\gamma}{2}}.
\end{align}
To estimate the last term on the right-hand side of~\eqref{est:Y-I}, we apply the Young's inequality for non-negative numbers $a$, $b$ and $\vartheta \in (0,1)$ as follows
\begin{align}\label{Young-inq}
 a^{\vartheta} b^{1-\vartheta} = \left(\varepsilon^{1-\frac{1}{\vartheta}} a\right)^{\vartheta} (\varepsilon b)^{1 - \vartheta}  \le \varepsilon^{1-\frac{1}{\vartheta}} a + \varepsilon b,
\end{align}
for every $\varepsilon>0$. To be more precise, let us choose $\varepsilon = \frac{1}{2}$, $\vartheta = \frac{p}{2} \in (\frac{1}{2},1)$, $a = 2^{\frac{(2-p)\gamma}{p}} [\Phi(v_1,v_2)]^{\frac{\gamma}{p}}$ and $b = |v_1-v_2|^{\gamma}$ in \eqref{Young-inq}, then it yields
\begin{align*}
2^{\frac{(2-p)\gamma}{2}} [\Phi(v_1,v_2)]^{\frac{\gamma}{2}} |v_1-v_2|^{\frac{(2-p)\gamma}{2}}  \le C [\Phi(v_1,v_2)]^{\frac{\gamma}{p}} + \frac{1}{2}  |v_1-v_2|^{\gamma}.
\end{align*}
Making use of this inequality, it follows from~\eqref{est:Y-I} that
\begin{align*}
\fint_{B} |v_1 - v_2|^{\gamma} dx & \le C \fint_{B}  [\Phi(v_1,v_2)]^{\frac{\gamma}{2}} |v_1|^{\frac{(2-p)\gamma}{2}} dx  + C \fint_{B}  [\Phi(v_1,v_2)]^{\frac{\gamma}{p}}dx \\
& \le C  \left(\fint_{B} [\Phi(v_1,v_2)]^{\frac{\gamma}{p}} dx\right)^{\frac{p}{2}} \left(\fint_{B} |v_1|^{\gamma} dx\right)^{1 - \frac{p}{2}} + C \fint_{B}  [\Phi(v_1,v_2)]^{\frac{\gamma}{p}}dx.
\end{align*}
Moreover, for any $s > \frac{\gamma}{p}$, applying H{\"o}lder's inequality 
\begin{align*}
\fint_{B}  [\Phi(v_1,v_2)]^{\frac{\gamma}{p}}dx \le \frac{s}{s-\frac{\gamma}{p}} [\mathcal{L}^n(B)]^{-\frac{\gamma}{ps}} \|\Phi(v_1,v_2)\|_{L^{s,\infty}(B)}^{\frac{\gamma}{p}},
\end{align*}
it completes that proof.
\end{proof}

\begin{lemma}\label{lem:Tech3}
Let $B \subset \Omega$ be an open set and two measurable functions $u,v$ such that $u-v \in L^{q,\infty}(B)$ for some $q>0$ . Assume that the function $f: \ \mathbb{R} \times \mathbb{R} \to \mathbb{R}^+$ satisfies
\begin{align}\label{cond:Tech3}
\int_{B \cap \left\{|u-v| \le k \right\}} f(u,v) dx \le \Pi k^{\vartheta},
\end{align}
for all $k >0$ with some constants $\Pi>0$, $\vartheta \ge 0$. Then there exists a constant $C = C(n,q,\vartheta)>0$ such that
\begin{align}\label{est:Tech3}
\|f(u,v)\|_{L^{\frac{q}{q+\vartheta},\infty}(B)} \le C \Pi \|u-v\|_{L^{q,\infty}(B)}^{\vartheta}.
\end{align}
\end{lemma}
\begin{proof}
For every $k$ and $\lambda>0$, let us consider the function
\begin{align*}
\Gamma(\lambda,k) = \mathcal{L}^n\left(B \cap \left\{f(u,v) > \lambda; \ |u-v|>k\right\}\right).
\end{align*}
It is readily verified that $\Gamma$ is non increasing with respect to the variable $\lambda$. Then,
\begin{align*}
\Gamma(\lambda,0) \le \frac{1}{\lambda} \int_0^{\lambda} \Gamma(t,0)dt \le \Gamma(0,k) + \frac{1}{\lambda} \int_0^{\lambda} \left(\Gamma(t,0) - \Gamma(t,k)\right) dt,
\end{align*}
under assumption~\eqref{cond:Tech3}, can be rewritten as
\begin{align}\nonumber
\mathcal{L}^n(B \cap \{f(u,v) > \lambda \}) & \le \mathcal{L}^n(B \cap \{|u-v|>k\}) + \frac{1}{\lambda} \int_{B \cap \{|u-v| \le k\}} f(u,v) dx \\ \label{est:B-7}
& \le C(n,q) k^{-q} \|u-v\|_{L^{q,\infty}(B)}^q + \frac{\Pi k^{\vartheta}}{\lambda}.
\end{align}
With a particular choice of $k$ as
\begin{align*}
k^{-q} \|u-v\|_{L^{q,\infty}(B)}^q = \frac{\Pi k^{\vartheta}}{\lambda} \  \Leftrightarrow  \ k = \left(\lambda \Pi^{-1} \|u-v\|_{L^{q,\infty}(B)}^q\right)^{\frac{1}{q+\vartheta}},
\end{align*}
we obtain from~\eqref{est:B-7} that
\begin{align}\label{est:B-8}
\lambda \mathcal{L}^n(B \cap \{f(u,v) > \lambda \})^{\frac{q+\vartheta}{q}} \le C  \Pi \|u-v\|_{L^{q,\infty}(B)}^{\vartheta}.
\end{align}
Taking the supremum both sides of~\eqref{est:B-8} for all $\lambda$ belonging to $(0,\infty)$, we plainly obtain~\eqref{est:Tech3}.
\end{proof}

With assumption \eqref{range-p} specified in this study, for the reader's convenience, let us introduce a new parameter
\begin{align}\label{p-tilde}
\tilde{p} = \min\left\{\frac{n}{2(n-1)}; \ \frac{(p-1)n}{n-p}\right\},
\end{align}
that is necessary for our proofs in this section.

\begin{lemma}\label{lem:B1}
Let $1<p\le 2-\frac{1}{n}$ and $B$ be a unit ball of $\Omega$. Assume that $u \in W^{1,p}(\Omega)$ and $v \in u + W_0^{1,p}(B)$ satisfying
\begin{align}\label{cond:lem-B1}
\int_{B \cap E_{k,h}} \Phi(\nabla u, \nabla v) dx \le \Pi k,
\end{align}
for some $\Pi>0$ and for every $k,h >0$, where the set $E_{k,h}$ defined by
\begin{align}\label{set-E}
E_{k,h} := \left\{x \in \Omega: \ h < |u - v| < k+h \right\}.
\end{align}
Then there exists $C = C(p,n)>0$ such that 
\begin{align}\label{est:lem-B1}
\|u-v\|_{L^{\tilde{p},\infty}(B)} & \le C  \Pi^{\frac{1}{p-1}} + C \Pi \int_{B} |\nabla u|^{2-p}dx.
\end{align}
Moreover, for every $\gamma \in \left(\gamma_{1}, \gamma_{2}\right)$ and $\kappa_1, \kappa_2 \in (0,1)$, there exists $C = C(p,n,\gamma,\kappa_1,\kappa_2)>0$ such that
\begin{align}\label{est-lem:u-v-B}
\fint_{B_{\varrho}} |\nabla u - \nabla v|^{\gamma} dx  & \le \kappa_1 \chi_1  \fint_{B_{\varrho}} |\nabla u|^{\gamma} dx  + \kappa_2 \chi_2\left(\fint_{B_{\varrho}} |\nabla u|^{2-p}dx\right)^{\frac{\gamma}{2-p}} + C \Pi^{\frac{\gamma}{p-1}},
\end{align} 
where $\chi_1, \chi_2$ are define as in~\eqref{def-chi-1}-\eqref{def-chi-2} and $\gamma_{1}, \gamma_{2}$ are given by~\eqref{p-star}.
\end{lemma}
\begin{proof}
For every $h >0$, let us set
\begin{align} \notag 
 F_{h} := \left\{x \in \Omega: \ |u - v| > h \right\}.
\end{align}
For all $k$ and $h>0$, it reveals the relationship between $E_{k,h}$, $F_{h}$ and $F_{k+h}$ that
\begin{align}\label{rel-EF}
E_{k,h} \subset  F_{h}, \quad F_{k+h} \subset F_{h} \quad \mbox{and} \quad E_{k,h} \cup F_{k+h} \subset F_{h}.
\end{align}
Moreover, one may check the following equality
\begin{align*}
\{|u-v|\ge k + h\} = \{|T_{k+h}(u-v)| \ge k + h\},
\end{align*}
which ensures that
\begin{align}\label{est-2:a}
(k+h) [\mathcal{L}^n(B \cap F_{k+h})]^{\frac{n-1}{n}} & \le  C \left[\int_{B \cap F_{k+h}} (T_{k+h}(u - v))^{\frac{n}{n-1}}dx\right]^{\frac{n-1}{n}}.
\end{align}
Due to the fact
\begin{align*}
|T_{k+h}(u-v)| \le \left(1+\frac{h}{k}\right) |T_{k+h}(u-v)-T_h(u-v)| \ \mbox{ in } F_h,
\end{align*}
and we combine with~\eqref{rel-EF}, \eqref{est-2:a} to discover that
\begin{align*}
(k+h) [\mathcal{L}^n(B \cap F_{k+h})]^{\frac{n-1}{n}} & \le  C \left[\int_{B \cap F_{h}} |T_{k+h}(u - v)- T_{h}(u-v)|^{\frac{n}{n-1}}dx\right]^{\frac{n-1}{n}} \notag \\
& \le C \left[\int_{B} |T_{k+h}(u - v) - T_{h}(u-v)|^{\frac{n}{n-1}}dx\right]^{\frac{n-1}{n}}.
\end{align*}
At this stage, it allows us to apply Sobolev's inequality to get
\begin{align}\label{est-2:lem-B1}
(k+h) [\mathcal{L}^n(B \cap F_{k+h})]^{\frac{n-1}{n}} &  \le C \int_{B \cap E_{k,h}} |\nabla u - \nabla v| dx.
\end{align}
Here, it is worth noticing that the constant $C$ in~\eqref{est-2:lem-B1} depends only on the ratio $\displaystyle{\frac{h}{k}}$. As an consequence of~\eqref{est:Y-I} in Lemma~\ref{lem:Tech2} with $\gamma = 1$, one finds
\begin{align*}
|\nabla u - \nabla v| \le C |\nabla u|^{\frac{2-p}{2}} [\Phi(\nabla u, \nabla v)]^{\frac{1}{2}} + C [\Phi(\nabla u, \nabla v)]^{\frac{1}{p}},
\end{align*}
which implies the following estimate 
\begin{align}\nonumber
(k+h) [\mathcal{L}^n(B \cap F_{k+h})]^{\frac{n-1}{n}} & \le  C \int_{B \cap E_{k,h}} |\nabla u|^{\frac{2-p}{2}} [\Phi(\nabla u, \nabla v)]^{\frac{1}{2}} dx \\ \label{est-3:lem-B1}
& \hspace{2cm} + C \int_{B \cap E_{k,h}} [\Phi(\nabla u, \nabla v)]^{\frac{1}{p}} dx.
\end{align}
Thanks to H{\"o}lder's inequality, we deduce from~\eqref{est-3:lem-B1} that 
\begin{align}\nonumber 
(k+h) [\mathcal{L}^n(B \cap F_{k+h})]^{\frac{n-1}{n}} & \le C \left(\int_{B} |\nabla u|^{2-p}dx\right)^{\frac{1}{2}} \left(\int_{B \cap E_{k,h}} \Phi(\nabla u, \nabla v) dx\right)^{\frac{1}{2}}  \\ \nonumber
& \hspace{1cm}  + C [\mathcal{L}^n(B \cap E_{k,h})]^{1 - \frac{1}{p}} \left(\int_{B \cap E_{k,h}} \Phi(\nabla u, \nabla v) dx\right)^{\frac{1}{p}},
\end{align}
and with assumption~\eqref{cond:lem-B1} in hand, it leads to
\begin{align}\label{est-4:lem-B1}
(k+h) [\mathcal{L}^n(B \cap F_{k+h})]^{\frac{n-1}{n}} & \le C (\Pi k)^{\frac{1}{2}} \left(\int_{B} |\nabla u|^{2-p}dx\right)^{\frac{1}{2}}  \notag \\
& \hspace{3cm} + C (\Pi k)^{\frac{1}{p}} [\mathcal{L}^n(B \cap E_{k,h})]^{1 - \frac{1}{p}}.
\end{align}
For all $\nu \ge 0$, there exists $m_{\nu}$ large enough such that $[\mathcal{L}^n(B\cap F_{k+h})]^{\nu} \le 1$ with $k = m_{\nu}h$. Moreover, it enables us to reuse~\eqref{rel-EF} to infer from~\eqref{est-4:lem-B1} that
\begin{align} \label{est-5a:lem-B1}
& \left((k+h) [\mathcal{L}^n(B \cap F_{k+h})]^{2\left(\frac{n-1}{n}+\nu\right)}\right)^{\frac{1}{2}}  \le  C \Pi^{\frac{1}{2}} \left(\int_{B} |\nabla u|^{2-p}dx\right)^{\frac{1}{2}}  \notag \\ 
& \hspace{6.5cm}  + C \Pi^{\frac{1}{p}} \left(m_{\nu}h[\mathcal{L}^n(B \cap F_{h})]^{\frac{2(p - 1 + \nu p)}{2-p}}\right)^{\frac{2-p}{2p}}.
\end{align}
Taking supremum for all $h$ belonging to $(0,\infty)$ both sides of~\eqref{est-5a:lem-B1}, one obtains
\begin{align}\label{est-5d:lem-B1}
& \|u-v\|_{L^{\frac{1}{2\left(\frac{n-1}{n}+\nu\right)},\infty}(B)}^{\frac{1}{2}}  \le  C \Pi^{\frac{1}{2}} \left(\int_{B} |\nabla u|^{2-p}dx\right)^{\frac{1}{2}}    + C \Pi^{\frac{1}{p}} \|u-v\|_{L^{\frac{2-p}{2(p - 1 + \nu p)},\infty}(B)}^{\frac{2-p}{2p}},
\end{align}
To obtain the same quasi-norm of Marcinkiewicz spaces in~\eqref{est-5d:lem-B1}, it allows us to choose $\nu$ satisfying
\begin{align}\label{cond-del}
\tilde{p} = \frac{1}{2\left(\frac{n-1}{n}+\nu\right)}  \ge \frac{2-p}{2(p - 1 + \nu p)},
\end{align}
where $\tilde{p}$ is defined as in~\eqref{p-tilde}. For this purpose, it is possible to choose
\begin{align*}
\nu = \begin{cases} \frac{3n-2 - p(2n-1)}{2n(p-1)}, &  \mbox{if } \ 1< p \le \frac{3n-2}{2n-1},\\ 0, & \mbox{if } \ \frac{3n-2}{2n-1} < p \le 2-\frac{1}{n}, \end{cases}
\end{align*}
which ensures the validity of~\eqref{cond-del}. With this choice of $\nu$, the inequality~\eqref{est-5d:lem-B1} yields that
\begin{align*}
\|u-v\|_{L^{\tilde{p},\infty}(B)}^{\frac{1}{2}} & \le  C \Pi^{\frac{1}{2}} \left(\int_{B} |\nabla u|^{2-p}dx\right)^{\frac{1}{2}}  + C \Pi^{\frac{1}{p}} \|u-v\|_{L^{\tilde{p},\infty}(B)}^{\frac{2-p}{2p}},
\end{align*}
which guarantees~\eqref{est:lem-B1} by taking Young's inequality into account.

Applying Lemma~\ref{lem:Tech3} with $f(u,v) = \Phi(\nabla u, \nabla v)$ and combining with estimate~\eqref{est:lem-B1}, it asserts 
\begin{align}\label{est:B-4-a}
\|\Phi(\nabla u, \nabla v)\|_{L^{\frac{\tilde{p}}{\tilde{p}+1},\infty}(B)} & \le C \Pi \|u-v\|_{L^{\tilde{p},\infty}(B)}  \le C \Pi \left(\Pi \int_{B} |\nabla u|^{2-p}dx +   \Pi^{\frac{1}{p-1}}\right).
\end{align}
It notices that $\gamma_{2} = \frac{p\tilde{p}}{\tilde{p}+1}$ from the combination of ~\eqref{p-tilde} and~\eqref{p-star}. Then, for every $0< \gamma < \gamma_{2}$ and $\kappa_1 \in (0,1)$, we can also apply Lemma~\ref{lem:Tech2} with $s = \frac{\tilde{p}}{\tilde{p}+1}$ to obtain
\begin{align}\label{est:B-5-a}
\fint_{B} |\nabla u - \nabla v|^{\gamma}dx \le \kappa_1 \fint_{B} |\nabla u|^{\gamma} dx + C  \|\Phi(\nabla u,\nabla v)\|_{L^{\frac{\tilde{p}}{\tilde{p}+1},\infty}(B)}^{\frac{\gamma}{p}}.
\end{align}
Substituting~\eqref{est:B-4-a} into~\eqref{est:B-5-a}, we have
\begin{align}\label{est-lem:u-v-B-2a}
 \fint_{B}  |\nabla u - \nabla v|^{\gamma}dx  & \le \kappa_1 \fint_{B} |\nabla u|^{\gamma} dx  + C  \left[\Pi^{\frac{2\gamma}{p}}\left(\fint_{B} |\nabla u|^{2-p}dx\right)^{\frac{\gamma}{p}} +  \Pi^{\frac{\gamma}{p-1}}\right].
\end{align}
It follows from Remark~\ref{rmk:1} that if $1<p\le \frac{3n-2}{2n-1}$ then $\gamma_{2} \le 2-p$. Thus, for every $\gamma \in (\gamma_{1}, \gamma_{2})$, thanks to H{\"o}lder's inequality, we infer that
\begin{align*}
\left(\fint_{B} |\nabla u|^{\gamma} dx \right)^{\frac{1}{\gamma}} \le \left(\fint_{B} |\nabla u|^{2-p} dx\right)^{\frac{1}{2-p}}.
\end{align*}
Finally, taking the comparison estimate in~\eqref{est-lem:u-v-B-2a} and Young's inequality for all $\kappa_2 \in (0,1)$ into account, it yields that 
\begin{align}\label{est-lem:u-v-B-2}
\fint_{B} |\nabla u - \nabla v|^{\gamma} dx  & \le \kappa_2 \left(\fint_{B} |\nabla u|^{2-p}dx\right)^{\frac{\gamma}{2-p}} + C \Pi^{\frac{\gamma}{p-1}}.
\end{align} 
Otherwise, when $\frac{3n-2}{2n-1}<p \le 2 - \frac{1}{n}$, it results $\gamma_{1} < \gamma_{2} < \frac{n(p-1)}{n-1}$ (see Remark~\ref{rmk:1}) and therefore, for every $\gamma \in (\gamma_{1}, \gamma_{2})$, the comparison estimate~\eqref{est-lem:u-v-B-2a} will be reduced to
\begin{align}\label{est-lem:u-v-B-3}
\fint_{B} |\nabla u - \nabla v|^{\gamma} dx & \le \kappa_1  \fint_{B} |\nabla u|^{\gamma}dx + C \Pi^{\frac{\gamma}{p-1}}.
\end{align}
Combining between~\eqref{est-lem:u-v-B-2} and~\eqref{est-lem:u-v-B-3} in two possible ranges of growth exponent $p$, the proof is complete.
\end{proof}

\subsection{Comparison to homogeneous problems}\label{sec:comparison}
Let us recall in the following lemma a basic property of the obstacle problem with frozen coefficients. Its proof is simple and can be found in \cite[Lemma 2.1]{Scheven2012}.
\begin{lemma}\label{lem:obs} 
Let $B$ be an open ball in $\Omega$ and $v \in \mathcal{T}^{1,p}(B)$ with $v \ge \psi$ on $\partial B$ be a weak solution to the following obstacle problem
\begin{align}\label{obs-frozen}
\int_{B} \langle \mathbb{A}(\nabla v,x), \nabla \varphi - \nabla v \rangle dx & \ge \int_{B} \langle \mathbb{A}(\nabla \psi,x), \nabla \varphi - \nabla v \rangle dx,
\end{align}
for every $\varphi \in \mathcal{T}^{1,p}_0(B)$. Then one has $v \ge \psi$ almost everywhere on $B$. 
\end{lemma}

\begin{lemma}\label{lem:comp1}
Let $1< p \le 2- \frac{1}{n}$, $\mu \in L^1(\Omega) \cap W^{-1,p'}(\Omega)$ and $u \in W^{1,p}(\Omega) \cap \mathcal{S}_0$ be a solution to obstacle problem $\mathcal{P}(\mathbb{A},\mu,\psi)$. Assume that $B_{\varrho}$ is an open ball in $\Omega$ with radius $\varrho>0$ and that $v \in u + W_0^{1,p}(B_{\varrho})$ is a weak solution to the following obstacle-free problem
\begin{align}\label{eq:test}
\begin{cases} -\mathrm{div}(\mathbb{A}(\nabla v,x)) & = -\mathrm{div}(\mathbb{A}(\nabla \psi,x)), \quad \mbox{ in } \ B_{\varrho}, \\
\hspace{1.3cm} v & = u, \hspace{2.8cm} \mbox{ on } \ \partial B_{\varrho}.\end{cases}
\end{align}
For every $\gamma \in (\gamma_{1},\gamma_{2})$ and $\kappa_1, \kappa_2 \in (0,1)$, there exists $C=C(\gamma,\kappa_1,\kappa_2,\mathtt{data})>0$ such that
\begin{align}\label{est:comp1}
\fint_{B_{\varrho}} |\nabla u - \nabla v|^{\gamma} dx & \le  \kappa_1 \chi_1 \fint_{B_{\varrho}} |\nabla u|^{\gamma} dx  + \kappa_2 \chi_2 \left(\fint_{B_{\varrho}}|\nabla u|^{2-p}dx\right)^{\frac{\gamma}{2-p}} \notag \\
 & \hspace{2cm} + C \left[\frac{|\mu|(B_{\varrho})}{\varrho^{n-1}} + \varrho\fint_{B_{\varrho}} \left|\mathrm{div}\left(\mathbb{A}(\nabla \psi,x)\right)\right|dx \right]^{\frac{\gamma}{p-1}},
\end{align} 
where $\chi_1, \chi_2$ are defined as in Lemma~\ref{lem:B1}.
\end{lemma}
\begin{proof} 
For simplicity, let us assume that $\varrho = 1$ and write $B$ instead of $B_{\varrho}$ in this proof. The proof of~\eqref{est-lem:u-v-B} for the ball $B_{\varrho}$ with radius $\varrho>0$ will be obtained by scaling. For every $k$, $h>0$, let us introduce the following truncation
\begin{equation}\label{eq_testfunction}
T_{k,h}(s) = \begin{cases} 0, & \ \mbox{ if } 0 \le |s| < h, \\  (|s| - h) \, \mathrm{sign}(s) , & \ \mbox{ if } h \le |s| \le k + h, \\ k  \, \mathrm{sign}(s),  & \ \mbox{ if } |s| > k + h. \end{cases}
\end{equation}
It is obviously $T_{k,h}(s) = T_{k}(s - T_{h}(s))$ which yields that 
\begin{align*}
u - T_{k,h}(u-v) \ge u - T_{k}(u-v) \ge v, \quad \mbox{ in } \{u \ge v\}.
\end{align*}
Moreover, as $v$ satisfies the obstacle problem~\eqref{obs-frozen}, due to Lemma~\ref{lem:obs}, one has $v \ge \psi$ a.e. in $B$. Thus, $u - T_{k,h}(u-v) \ge \psi$ a.e. in $B$. For this reason, we may substitute $\varphi =  u - T_{k,h}(u-v)$ to the variational inequality~\eqref{eq:main} to obtain that
\begin{align}\label{est:B-1}
\int_{B} \langle \mathbb{A}(\nabla u,x), \nabla T_{k,h}(u-v) \rangle dx \le  \int_{B} T_{k,h}(u-v) d\mu.
\end{align}
On the other hand, since $v$ solves~\eqref{eq:test}, there holds
\begin{align}\label{eq:homogeneous-B}
\int_{B} \langle \mathbb{A}(\nabla v,x), \nabla \phi \rangle dx = \int_{B}  - \mathrm{div}\left(\mathbb{A}(\nabla \psi,x)\right)\phi  dx, 
\end{align}
for all $\phi \in W_0^{1,p}(B)$. Let us invoke $\phi =  T_{k,h}(u-v)$ in~\eqref{eq:homogeneous-B}, then
\begin{align}\label{est:B-2}
\int_{B} \langle \mathbb{A}(\nabla v,x), \nabla T_{k,h}(u-v) \rangle dx = \int_{B}  - \mathrm{div}\left(\mathbb{A}(\nabla \psi,x)\right) T_{k,h}(u-v)  dx, 
\end{align}
Subtracting two inequalities~\eqref{est:B-1} and~\eqref{est:B-2}, it enables us to get
\begin{align}\label{est:B-3}
& \int_{B} \langle \mathbb{A}(\nabla u,x) - \mathbb{A}(\nabla v,x), \nabla T_{k,h}(u-v) \rangle dx \le  \int_{B} T_{k,h}(u-v) d\mu \notag \\
& \hspace{6cm} + \int_{B} \mathrm{div}\left(\mathbb{A}(\nabla \psi,x)\right) T_{k,h}(u-v)  dx.
\end{align}
At this stage, taking the advantage of notation $E_{k,h}$ in~\eqref{set-E} and conditions in~\eqref{cond:A}, we deduce from~\eqref{est:B-3} that
\begin{align*}
\int_{B\cap E_{k,h}} \Phi(\nabla u, \nabla v) dx \le \Upsilon k \left(|\mu|(B) +  \int_{B} \left|\mathrm{div}\left(\mathbb{A}(\nabla \psi,x)\right)\right|dx\right). 
\end{align*} 
Following the argument of Lemma~\ref{lem:B1}, for every $\kappa_1, \kappa_2 \in (0,1)$ one has
\begin{align}\notag
 \fint_{B} |\nabla u - \nabla v|^{\gamma} dx & \le  \kappa_1 \chi_1 \fint_{B} |\nabla u|^{\gamma} dx  + \kappa_2 \chi_2 \left(\fint_{B}|\nabla u|^{2-p}dx\right)^{\frac{\gamma}{2-p}}\notag \\
 & \hspace{1cm} + C \left[|\mu|(B) + \int_{B} \left|\mathrm{div}\left(\mathbb{A}(\nabla \psi,x)\right)\right|dx \right]^{\frac{\gamma}{p-1}}. \notag
\end{align}
which implies to~\eqref{est:comp1} by scaling.
\end{proof}

\begin{lemma}\label{lem:comp2}
Let $1< p \le 2-\frac{1}{n}$, $\mu \in L^1(\Omega) \cap W^{-1,p'}(\Omega)$ and $u \in W^{1,p}(\Omega) \cap \mathcal{S}_0$ be a solution to obstacle problem $\mathcal{P}(\mathbb{A},\mu,\psi)$. Assume $B_{\varrho}$ is an open ball in $\Omega$ with radius $\varrho>0$. There exists $\tilde{u} \in W^{1,p}(B_{\varrho/4}) \cap W^{1,\infty}(B_{\varrho/8})$ such that for every $\gamma \in (\gamma_{1},\gamma_{2})$ and $\kappa, \kappa_1, \kappa_2 \in (0,1)$, there holds
\begin{align}\label{est:comp2-a}
\|\nabla \tilde{u}\|_{L^{\infty}(B_{\varrho/8})} & \le C\left(\fint_{B_{\varrho/4}} |\nabla u|^{\gamma} dx \right)^{\frac{1}{\gamma}}  + \kappa \chi_2 \left(\fint_{B_{\varrho/4}}|\nabla u|^{2-p}dx\right)^{\frac{1}{2-p}} \notag \\
& \hspace{2cm} + C \left[\frac{|\mu|(B_{\varrho/4})}{\varrho^{n-1}} + \varrho\fint_{B_{\varrho/4}} \left|\mathrm{div}\left(\mathbb{A}(\nabla \psi,x)\right)\right|dx \right]^{\frac{1}{p-1}},
\end{align}
and
\begin{align}\label{est:comp2-b}
\fint_{B_{\varrho/4}} |\nabla u - \nabla \tilde{u}|^{\gamma} dx & \le  \kappa_1 \chi_1 \fint_{B_{\varrho/4}} |\nabla u|^{\gamma} dx + (\kappa_2 +\delta)\chi_2 \left(\fint_{B_{\varrho/4}}|\nabla u|^{2-p}dx\right)^{\frac{\gamma}{2-p}} \notag \\
 & \hspace{2cm} + C \left[\frac{|\mu|(B_{\varrho/4})}{\varrho^{n-1}} + \varrho\fint_{B_{\varrho/4}} \left|\mathrm{div}\left(\mathbb{A}(\nabla \psi,x)\right)\right|dx \right]^{\frac{\gamma}{p-1}},
\end{align} 
if provided $[\mathbb{A}]^{r_0} \le \delta$, where $C = C(\gamma,\kappa_1,\kappa_2,\mathtt{data})>0$.
\end{lemma}
\begin{proof}
Let $v \in u + W_0^{1,p}(B_{\varrho})$ be a weak solution to~\eqref{eq:test}. The argument of Lemma~\ref{lem:comp1} gives us
\begin{align}\label{est-I}
\fint_{B_{\varrho}} |\nabla u - \nabla v|^{\gamma} dx \ & \le  \kappa_1 \chi_1 \fint_{B_{\varrho}} |\nabla u|^{\gamma} dx   + \kappa_2 \chi_2 \left(\fint_{B_{\varrho}}|\nabla u|^{2-p}dx\right)^{\frac{\gamma}{2-p}} \notag \\
 & \hspace{2cm} + C \left[\frac{|\mu|(B_{\varrho})}{\varrho^{n-1}} + \varrho\fint_{B_{\varrho}} \left|\mathrm{div}\left(\mathbb{A}(\nabla \psi,x)\right)\right|dx \right]^{\frac{\gamma}{p-1}},
\end{align}
for every $\kappa_1, \kappa_2 \in (0,1)$. Assume further that $w$ solves the following problem
\begin{align}\label{eq:comp2}
\begin{cases} -\mathrm{div}(\mathbb{A}(\nabla w,x)) & = 0, \quad \mbox{ in } \ B_{\varrho/2}, \\
\hspace{1cm} w & = v, \quad \mbox{ on } \ \partial B_{\varrho/2}.\end{cases}
\end{align}
Similar to the argument as in proof of Lemma~\eqref{lem:comp1}, for $B_{\varrho/2} \equiv B$, one obtains that
\begin{align}\notag
\int_{B\cap E_{k,h}} \Phi(\nabla v, \nabla w) dx \le \Upsilon k \int_{B} \left|\mathrm{div}\left(\mathbb{A}(\nabla \psi,x)\right)\right|dx, \quad \forall h, k>0,
\end{align}
from which Lemma~\ref{lem:B1} and scaling follow, it is sufficient to prove that
\begin{align}\label{est-II}
 & \fint_{B_{\varrho/2}} |\nabla v - \nabla w|^{\gamma} dx  \le  \kappa_3 \chi_1 \fint_{B_{\varrho/2}} |\nabla v|^{\gamma} dx  \notag \\
 & \hspace{1cm} + \kappa_4 \chi_2 \left(\fint_{B_{\varrho/2}}|\nabla v|^{2-p}dx\right)^{\frac{\gamma}{2-p}} + C \left[\varrho\fint_{B_{\varrho/2}} \left|\mathrm{div}\left(\mathbb{A}(\nabla \psi,x)\right)\right|dx \right]^{\frac{\gamma}{p-1}},
\end{align} 
for every $\kappa_3, \kappa_4 \in (0,1)$. Next, let $\tilde{u}$ be the unique solution to
\begin{align}\label{eq:comp3}
\begin{cases} -\mathrm{div}(\overline{\mathbb{A}}_{B_{\varrho/4}}(\nabla \tilde{u})) & = 0, \quad \mbox{ in } \ B_{\varrho/4}, \\
\hspace{1cm} \tilde{u} & = w, \quad \mbox{on } \ \partial B_{\varrho/4}.\end{cases}
\end{align}
In the same manner as proof of~\cite[Proposition 2.3]{55QH4}, under the assumption $[\mathbb{A}]^{r_0} \le \delta$, it allows us to show that
\begin{align*}
\|\nabla \tilde{u}\|_{L^{\infty}(B_{\varrho/8})} \le C \left(\fint_{B_{\varrho/4}}|\nabla w|^pdx\right)^{\frac{1}{p}},
\end{align*}
and 
\begin{align*}
\fint_{B_{\varrho/4}}|\nabla \tilde{u} - \nabla w|dx \le C \delta \left(\fint_{B_{\varrho/4}}|\nabla w|^pdx\right)^{\frac{1}{p}}.
\end{align*}
As a consequence of the reverse H\"older inequality to $w$, the comparison estimates~\eqref{est-I} and~\eqref{est-II}, it establishes our conclusion for both~\eqref{est:comp2-a} and~\eqref{est:comp2-b} by choosing suitable values of $\kappa_3$ and $\kappa_4$. 
\end{proof}

In order to obtain the comparison estimates up to the boundary, one needs an extra regularity assumption on the boundary. And here, the condition $\mathcal{H}_{r_0,\delta}$ is assumed on the coefficient $(\mathbb{A};\Omega)$, for some $r_0>0$ and $\delta \in (0,1/2)$. Under this assumption, in a completely similar argument as in Lemma~\ref{lem:comp2}, it is possible to claim the boundary comparison estimates. We also refer to the proof of~\cite[Proposition 2.6]{55QH4} to explore more.

\begin{lemma}\label{lem:comp3}
Let $1< p \le 2-\frac{1}{n}$, $\mu \in L^1(\Omega) \cap W^{-1,p'}(\Omega)$ and $u \in W^{1,p}(\Omega) \cap \mathcal{S}_0$ be a solution to obstacle problem $\mathcal{P}(\mathbb{A},\mu,\psi)$. Assume that $(\mathbb{A};\Omega) \in \mathcal{H}_{r_0,\delta}$ for some $r_0>0$ and $\delta \in (0,1/2)$. For each $\xi \in \partial\Omega$ and $0 < \varrho \le r_0$, there exists $\tilde{u} \in W^{1,p}(\Omega_{\varrho/10}) \cap W^{1,\infty}(\Omega_{\varrho/100})$ such that for every $\gamma \in (\gamma_{1},\gamma_{2})$ and $\kappa, \kappa_1, \kappa_2 \in (0,1)$, there holds
\begin{align}\label{est:comp3-a}
\|\nabla \tilde{u}\|_{L^{\infty}(\Omega_{\varrho/100})} & \le C\left(\fint_{\Omega_{\varrho/10}} |\nabla u|^{\gamma} dx \right)^{\frac{1}{\gamma}}  + \kappa \chi_2 \left(\fint_{\Omega_{\varrho/10}}|\nabla u|^{2-p}dx\right)^{\frac{1}{2-p}} \notag \\
& \hspace{2cm} + C \left[\frac{|\mu|(\Omega_{\varrho/10})}{\varrho^{n-1}} + \varrho\fint_{\Omega_{\varrho/10}} \left|\mathrm{div}\left(\mathbb{A}(\nabla \psi,x)\right)\right|dx \right]^{\frac{1}{p-1}},
\end{align}
and
\begin{align}\label{est:comp3-b}
\fint_{\Omega_{\varrho/10}} |\nabla u - \nabla \tilde{u}|^{\gamma} dx &\le  \kappa_1 \chi_1 \fint_{\Omega_{\varrho/10}} |\nabla u|^{\gamma} dx + (\kappa_2 +\delta)\chi_2 \left(\fint_{\Omega_{\varrho/10}}|\nabla u|^{2-p}dx\right)^{\frac{\gamma}{2-p}} \notag \\
 & \hspace{1cm} + C \left[\frac{|\mu|(B_{\varrho/10})}{\varrho^{n-1}} + \varrho\fint_{\Omega_{\varrho/10}} \left|\mathrm{div}\left(\mathbb{A}(\nabla \psi,x)\right)\right|dx \right]^{\frac{\gamma}{p-1}}.
\end{align} 
Here we denote $\Omega_{\varrho} = B_{\varrho}(\xi) \cap \Omega$ and the constant $C = C(\gamma,\kappa_1,\kappa_2,\mathtt{data})>0$.
\end{lemma}

\section{Level-set decay estimate}
\label{sec:levelset}
In this section, we consider $u$ as a solution to problem $\mathcal{P}(\mathbb{A},\mu,\psi)$ satisfying $|\nabla u| \in L^{2-p}(\Omega)$ with $1<p\le 2-\frac{1}{n}$ and given data $\mu \in \mathcal{M}_b(\Omega)$ and $\mathrm{div}\left(\mathbb{A}(\nabla \psi,\cdot)\right) \in L^1(\Omega)$. Moreover, we assume that $(\mathbb{A};\Omega) \in \mathcal{H}_{r_0,\delta}$ for $r_0>0$ and $\delta>0$ small enough. For the sake of readability, for some $\lambda>0$ and $\varepsilon \in (0,1)$, let us introduce the following subsets of $\Omega$:
\begin{align}\label{def:VW}
\begin{cases}
\mathbb{V} = \mathbb{V}_1 \setminus \left(\mathbb{V}_2 \cup \mathbb{V}_3\right), \\
\mathbb{V}_1 = \left\{\mathbf{M}_{\alpha}(|\nabla u|^{\gamma})> a\lambda\right\}, \\ 
\mathbb{V}_2 = \begin{cases} \left\{\left[\mathbf{M}_{\sigma}(|\nabla u|^{2-p})\right]^{\frac{\gamma}{2-p}} > \varepsilon^{-\gamma}\lambda \right\},  \ \mbox{ if } 1< p \le \frac{3n-2}{2n-1}, \\ \emptyset,  \ \mbox{ if } \frac{3n-2}{2n-1}< p \le 2 - \frac{1}{n}, \end{cases} \\
\mathbb{V}_3 = \left\{\left[\mathbf{M}_{\beta}(\mu) + \mathbf{M}_{\beta}\left(\mathrm{div}\left(\mathbb{A}(\nabla \psi,\cdot)\right)\right)\right]^{\frac{\gamma}{p-1}} >  \varepsilon^2 \lambda\right\}, \\ 
\mathbb{W} = \left\{\mathbf{M}_{\alpha}(|\nabla u|^{\gamma})> \lambda\right\}.
\end{cases}
\end{align}

\begin{lemma}\label{lem:step1}
Let $D_0=\mathrm{diam}(\Omega)$, $0<R_0<r_0$ and $a>0$. One can find $\varepsilon_0 = \varepsilon_0(a,\tilde{\mathtt{data}},D_0/R_0) \in (0,1)$ such that 
$\mathcal{L}^n \left(\mathbb{V}\right) <  \varepsilon \mathcal{L}^n ({B_{R_0}}(0))$, for every $\lambda>0$ and $\varepsilon \in (0,\varepsilon_0)$.
\end{lemma}
\begin{proof} It is certainly pleasing when $\mathbb{V} = \emptyset$. On the other hand, for the case there is $\xi \in \Omega$ such that $\left[\mathbf{M}_{\beta}(\mu)(\xi)\right]^{\frac{\gamma}{p-1}}\le  \varepsilon^2 \lambda$, it follows that
\begin{align}\label{3.4}
|\mu|(\Omega) \le {|\mu|(B_{D_0}(\xi))} \le C{D_0^{n-\beta}} \mathbf{M}_{\beta}(\mu)(\xi) \le CD_0^{n-\beta}(\varepsilon^2\lambda)^{\frac{p-1}{\gamma}}.
\end{align}
Moreover, since $u \in \mathcal{S}_0$ is a limit of approximating solutions to problem $\mathcal{P}(\mathbb{A},\mu,\psi)$, there exist a sequence $(\mu_k)_{k \in \mathbb{N}} \subset L^1(\Omega) \cap W^{-1,p'}(\Omega)$ with $\mu_k \to \mu$ in the narrow topology of measures in $\mathcal{M}_b(\Omega)$ and a sequence of solutions $u_k \in W^{1,p}(\Omega) \cap \mathcal{S}_0$ to the variational inequality~\eqref{eq:var-form} for all $\varphi \in u_k + W^{1,p}_0(\Omega) \cap \mathcal{S}_0$, such that three claims of~\eqref{conver} are well satisfied. On the other hand, since $0<\gamma < \gamma_{2} \le \frac{n(p-1)}{n-1}$, due to~\cite[Lemma 3.3]{Scheven2012}, it is known that 
\begin{align*}
\left(\dfrac{1}{D_0^n}\int_{\Omega}|\nabla u_k|^{\gamma}dx\right)^{\frac{1}{\gamma}}\le C({\gamma})\left[\dfrac{|\mu_k|(\Omega)}{D_0^{n-1}}\right]^{\frac{1}{p-1}},
\end{align*}
for every $k \in \mathbb{N}$. As a consequence, it comes out 
\begin{align*}
\left(\dfrac{1}{D_0^n}\int_{\Omega}|\nabla u|^{\gamma}dx\right)^{\frac{1}{\gamma}} & \le C({\gamma}) \left(\dfrac{1}{D_0^n}\int_{\Omega}|\nabla u_k|^{\gamma}dx\right)^{\frac{1}{\gamma}} + C({\gamma}) \left(\dfrac{1}{D_0^n}\int_{\Omega}|\nabla u - \nabla u_k|^{\gamma}dx\right)^{\frac{1}{\gamma}} \\
& \le C({\gamma})\left[\dfrac{|\mu_k|(\Omega)}{D_0^{n-1}}\right]^{\frac{1}{p-1}} +  C({\gamma}) \left(\dfrac{1}{D_0^n}\int_{\Omega}|\nabla u - \nabla u_k|^{\gamma}dx\right)^{\frac{1}{\gamma}}.
\end{align*}
Passing to the limit when $k \to \infty$. Note that $\nabla u \to \nabla u_k$ in $L^{\gamma}(\Omega)$, we are allowed to obtain
\begin{align*}
\left(\dfrac{1}{D_0^n}\int_{\Omega}|\nabla u|^{\gamma}dx\right)^{\frac{1}{\gamma}}\le C({\gamma})\left[\dfrac{|\mu|(\Omega)}{D_0^{n-1}}\right]^{\frac{1}{p-1}},
\end{align*}
which implies from~\eqref{3.4} that
\begin{align}\label{3.5}
\int_{\Omega} |\nabla u|^{\gamma}dx \le C D_0^n \left[\dfrac{D_0^{n-\beta}(\varepsilon^2\lambda)^{\frac{p-1}{\gamma}}}{D_0^{n-1}}\right]^{\frac{\gamma}{p-1}}  \le  C D_0^{n-\frac{(\beta-1)\gamma}{p-1}} \varepsilon^2\lambda.
\end{align}
Applying the boundedness property of $\mathbf{M}_{\alpha}$ from $L^1(\mathbb{R}^n)$ into $L^{1,\infty}(\mathbb{R}^n)$, it yields from~\eqref{3.5} that
\begin{align}\label{eq:est100}
\mathcal{L}^n \left(\mathbb{V}\right) \le \mathcal{L}^n \left(\mathbb{V}_1\right)  \le \left[\frac{C}{a\lambda} \int_{\Omega} |\nabla u|^{\gamma}dx\right]^{\frac{n}{n-\alpha}} \le C_1 \left(a^{-1}\varepsilon^2\right)^{\frac{n}{n-\alpha}} \mathcal{L}^n ({B_{R_0}}(0)).
\end{align}
It is worth emphasizing that the last constant $C_1$ depends not only on $\tilde{\mathtt{data}}$ but also the ratio $D_0/R_0$. Furthermore, it is interesting to note that 
$$\left[n-\frac{(\beta-1)\gamma}{p-1}\right]\frac{n}{n-\alpha} = n.$$ 
Then, for $a>0$ we may choose  $\varepsilon_0 \in (0,1)$ such that 
\begin{align}\notag
C_1 (a^{-1}\varepsilon_0^2)^{\frac{n}{n-\alpha}} < \varepsilon_0 \Leftrightarrow \varepsilon_0 < C_1^{\frac{\alpha-n}{\alpha+n}} a^{\frac{n}{\alpha+n}},
\end{align}
which in turn will allow us to conclude the proof of lemma for all $\varepsilon \in (0,\varepsilon_0)$.
\end{proof}

\begin{lemma}\label{lem:cutoff}
Let $x \in \Omega$ and $R>0$ satisfying $B_R(x) \cap \Omega \not\subset \mathbb{W}$. Then the following inequality
\begin{equation}\label{eq:cutoff}
\mathcal{L}^n \left(\mathbb{V} \cap B_{R}(x)\right) \leq  \mathcal{L}^n \left(\left\{\mathbf{M}_{\alpha}^R(|\nabla u|^{\gamma}) > a\lambda\right\} \cap B_{R}(x)\right),  
\end{equation}
holds for any $a>3^{n-\alpha}$ and $\lambda>0$. Here the cut-off operator $\mathbf{M}_{\alpha}^R$ is given by
\begin{align}\label{def:M-R}
\mathbf{M}_{\alpha}^R(|\nabla u|^{\gamma})(y) := \sup_{0<\varrho_1<R} \varrho_1^{\alpha} \fint_{B_{\varrho_1}(y)} |\nabla u|^{\gamma}(z) dz.
\end{align}
\end{lemma}
\begin{proof}
For every $y \in B_R(x)$, let us decompose the classical fractional maximal operator as the form of the cut-off ones as 
\begin{align}\label{cutoff-1}
\mathbf{M}_{\alpha}(|\nabla u|^{\gamma})(y) = \max\left\{\mathbf{M}_{\alpha}^R(|\nabla u|^{\gamma})(y); \ \mathbf{T}_{\alpha}^R(|\nabla u|^{\gamma})(y) \right\},
\end{align}
where $\mathbf{M}_{\alpha}^R$ is defined in~\eqref{def:M-R} and $\mathbf{T}_{\alpha}^R$ is given by
\begin{align} \notag 
\mathbf{T}_{\alpha}^R(|\nabla u|^{\gamma})(y):= \sup_{\varrho_2 \ge R} \varrho_2^{\alpha} \fint_{B_{\varrho_2}(y)} |\nabla u|^{\gamma}(z) dz.
\end{align}
Since $B_R(x) \cap \Omega \not\subset \mathbb{W}$, then there exists $\xi \in B_R(x)$ such that $\mathbf{M}_{\alpha}(|\nabla u|^{\gamma})(\xi) \le \lambda$. Moreover, since $y, \xi \in B_R(x)$, we infer that for any $\varrho_2 \ge R$
\begin{align*}
B_{\varrho_2}(y) \subset B_{\varrho_2 + R}(x) \subset B_{\varrho_2 + 2R}(\xi) \subset B_{3\varrho_2}(\xi),
\end{align*}
which leads to
\begin{align}\label{cutoff-2}
\mathbf{T}_{\alpha}^R(|\nabla u|^{\gamma})(y) \le 3^n \sup_{\varrho_2 \ge R} \varrho_2^{\alpha} \fint_{B_{3\varrho_2}(\xi)} |\nabla u|^{\gamma}(z) dz \le 3^{n-\alpha} \mathbf{M}_{\alpha}(|\nabla u|^{\gamma})(\xi) \le 3^{n-\alpha} \lambda^{\gamma}.
\end{align}
Substituting~\eqref{cutoff-2} into~\eqref{cutoff-1}, one gets that
\begin{align}\nonumber
\mathbf{M}_{\alpha}(|\nabla u|^{\gamma})(y) \le \max\left\{\mathbf{M}_{\alpha}^R(|\nabla u|^{\gamma})(y); \ 3^{n-\alpha} \lambda \right\},
\end{align}
for all $ y \in B_R(x)$. Therefore, we are able to conclude~\eqref{eq:cutoff} for any $a>3^{n-\alpha}$. 
\end{proof}

\begin{lemma}\label{lem:step2}
Let $x \in \Omega$ and $R>0$ satisfying $B_R(x) \cap \Omega \not\subset \mathbb{W}$. There exist $a=a(\tilde{\mathtt{data}})>0$, $\delta = \delta(\tilde{\mathtt{data}}) \in (0,1/2)$ and $\varepsilon_0 = \varepsilon_0(\tilde{\mathtt{data}}) \in (0,1)$ such that if $(\mathbb{A};\Omega) \in \mathcal{H}_{r_0,\delta}$ for $r_0>0$ then the following inequality
\begin{equation}\label{eq:step2}
\mathcal{L}^n \left(\left\{\mathbf{M}_{\alpha}^R(|\nabla u|^{\gamma}) > a\lambda\right\} \cap B_{R}(x)\right) \le \varepsilon \mathcal{L}^n \left(B_{R}(x)\right),  
\end{equation}
holds for any $\lambda>0$ and $\varepsilon \in (0,\varepsilon_0)$.
\end{lemma}
\begin{proof}
Since $B_R(x) \cap \Omega \not \subset \mathbb{W}$,  there exists $\xi_1 \in B_{R}(x) \cap \Omega$ such that 
\begin{align}\label{cond:step2}
\mathbf{M}_{\alpha}(|\nabla u|^{\gamma})(\xi_1) \le \lambda. 
\end{align}
Without loss of generality, we can assume $\mathbb{V} \cap B_R(x) \neq \emptyset$. Then, there exists $\xi_2 \in B_{R}(x) \cap \Omega$ such that 
\begin{align}\label{cond:step2-b}
\chi_2\left[\mathbf{M}_{\sigma}(|\nabla u|^{2-p})(\xi_2)\right]^{\frac{\gamma}{2-p}} \le \varepsilon^{-\gamma}\lambda,  
\end{align}
and
\begin{align}\label{cond:step2-c}
\left[\mathbf{M}_{\beta}(\mu)(\xi_2) + \mathbf{M}_{\beta}\left(\mathrm{div}\left(\mathbb{A}(\nabla \psi,\cdot)\right)\right)(\xi_2)\right]^{\frac{\gamma}{p-1}} \le  \varepsilon^2 \lambda. 
\end{align}
One also notices that $\mathbb{V}_2 = \emptyset$ when $\chi_2=0$ and it makes the estimate~\eqref{cond:step2-b} valid. Similar to the proof of Lemma~\ref{lem:step1}, we can find a sequence $(\mu_k)_{k \in \mathbb{N}} \subset L^1(\Omega) \cap W^{-1,p'}(\Omega)$ with $\mu_k \to \mu$ in the narrow topology of measures in $\mathcal{M}_b(\Omega)$ and solutions $u_k \in W^{1,p}(\Omega) \cap \mathcal{S}_0$ of the variational inequality~\eqref{eq:var-form} for all $\varphi \in u_k + W^{1,p}_0(\Omega) \cap \mathcal{S}_0$, such that the claims in~\eqref{conver} are well satisfied. At this stage, we separate the proof into two cases:  $B_{16R}(x) \subset \Omega$ and $B_{16R}(x) \cap \partial \Omega \neq \emptyset$. 

To handle the first case when $B_{16R}(x) \subset \Omega$, invoking Lemma~\ref{lem:comp2}, there exists $\tilde{u}_k \in W^{1,\infty}(B_{2R}(x)) \cap W^{1,p}(B_{4R}(x))$ such that if $[\mathbb{A}]^{r_0} \le \delta$, then for every $\kappa, \kappa_1, \kappa_2 \in (0,1)$, it holds
\begin{align}\label{est:uk-inf}
\|\nabla \tilde{u}_k\|_{L^{\infty}(B_{2R}(x))}  & \le C \left(\fint_{B_{4R}(x)} |\nabla u_k|^{\gamma} dz\right)^{\frac{1}{\gamma}}  + \kappa \chi_2 \left[\fint_{B_{4R}(x)} |\nabla u_k|^{2-p}dz\right]^{\frac{1}{2-p}} \notag \\
&  + C \left[\frac{|\mu_k|(B_{4R}(x))}{R^{n-1}} + R\fint_{B_{4R}(x)} \left|\mathrm{div}\left(\mathbb{A}(\nabla \psi,z)\right)\right|dz\right]^{\frac{1}{p-1}},
\end{align}
and
\begin{align}\label{est:uk-vk}
 \fint_{B_{4R}(x)} |\nabla u_k - \nabla \tilde{u}_k|^{\gamma} dz  & \le \kappa_1 \chi_1 \fint_{B_{4R}(x)} |\nabla u_k|^{\gamma} dz  + (\kappa_2 +\delta)\chi_2 \left[\fint_{B_{4R}(x)} |\nabla u_k|^{2-p}dz\right]^{\frac{\gamma}{2-p}} \notag \\
&  + C \left[\frac{|\mu_k|(B_{4R}(x))}{R^{n-1}} + R\fint_{B_{4R}(x)} \left|\mathrm{div}\left(\mathbb{A}(\nabla \psi,z)\right)\right|dz\right]^{\frac{\gamma}{p-1}},
\end{align} 
where $\chi_1$ and $\chi_2$ are defined as in Lemma~\ref{lem:B1}. The left-hand side of~\eqref{eq:step2} can be estimated by the sum of three separate terms as follows
\begin{align}\label{dec-N}
\mathcal{N} := \mathcal{L}^n \left(\left\{\mathbf{M}_{\alpha}^R(|\nabla u|^{\gamma}) > a\lambda\right\} \cap B_{R}(x)\right) \le \mathcal{N}_1 + \mathcal{N}_2 + \mathcal{N}_3, 
\end{align}
where 
\begin{align*}
\mathcal{N}_1 &:= \mathcal{L}^n \left(\left\{\mathbf{M}_{\alpha}^R\left(\chi_{B_{2R}(x)}|\nabla u - \nabla u_k|^{\gamma}\right) > a\lambda/3\right\} \cap B_{R}(x)\right),\\
\mathcal{N}_2 &:= \mathcal{L}^n \left(\left\{\mathbf{M}_{\alpha}^R\left(\chi_{B_{2R}(x)}|\nabla u_k - \nabla \tilde{u}_k|^{\gamma}\right) > a\lambda/3\right\} \cap B_{R}(x)\right),\\
\mathcal{N}_3 &:= \mathcal{L}^n \left(\left\{\mathbf{M}_{\alpha}^R\left(\chi_{B_{2R}(x)}|\nabla \tilde{u}_k|^{\gamma}\right) > a\lambda/3\right\} \cap B_{R}(x)\right).
\end{align*}
Since $\nabla u \to \nabla u_k$ in $L^{\gamma}(\Omega)$ as $k \to \infty$, it is readily to conclude the first term $\mathcal{N}_1 \to 0$. We now show that the third term $\mathcal{N}_3$ also vanishes as $k \to \infty$. For each $y \in B_{R}(x)$, since $B_{\varrho}(y) \subset B_{2R}(x)$ for all $\varrho \in (0,R)$, the inequality~\eqref{est:uk-inf} yields that
\begin{align}
\mathbf{M}_{\alpha}^R\left(\chi_{B_{2R}(x)}|\nabla \tilde{u}_k|^{\gamma}\right)(y) & = \sup_{0<\varrho<R} \left[\varrho^{\alpha} \fint_{B_{\varrho}(y)} \chi_{B_{2R}(x)}|\nabla \tilde{u}_k|^{\gamma} dz\right]  \le C R^{\alpha} \|\nabla \tilde{u}_k\|_{L^{\infty}(B_{2R}(x))}^{\gamma} \notag \\
& \le C R^{\alpha}\fint_{B_{4R}(x)} |\nabla u_k|^{\gamma} dz  + \kappa \chi_2 C R^{\alpha}\left[\fint_{B_{4R}(x)} |\nabla u_k|^{2-p}dz\right]^{\frac{\gamma}{2-p}} \notag \\
&  + C R^{\alpha}\left[\frac{|\mu_k|(B_{4R}(x))}{R^{n-1}} + R\fint_{B_{4R}(x)} \left|\mathrm{div}\left(\mathbb{A}(\nabla \psi,z)\right)\right|dz\right]^{\frac{\gamma}{p-1}},\notag
\end{align}
which immediately leads to
\begin{align}\label{est:uk-inf-2}
\limsup_{k \to \infty} \mathbf{M}_{\alpha}^R\left(\chi_{B_{2R}(x)}|\nabla \tilde{u}_k|^{\gamma}\right)(y)  & \le C R^{\alpha} \fint_{B_{4R}(x)} |\nabla u|^{\gamma} dz + \kappa \chi_2 C R^{\alpha} \left[\fint_{B_{4R}(x)} |\nabla u|^{2-p}dz\right]^{\frac{\gamma}{2-p}} \notag \\
&  + C R^{\alpha} \left[\frac{|\mu|(B_{4R}(x))}{R^{n-1}} + R\fint_{B_{4R}(x)} \left|\mathrm{div}\left(\mathbb{A}(\nabla \psi,z)\right)\right|dz\right]^{\frac{\gamma}{p-1}}.
\end{align}
Moreover, the existence of $\xi_1 \in B_{R}(x)$ satisfying~\eqref{cond:step2} ensures that $B_{4R}(x) \subset B_{5R}(\xi_1)$, it results that
\begin{align}\label{est:111}
R^{\alpha}\fint_{B_{4R}(x)} |\nabla u|^{\gamma} dz \le (5/4)^n R^{\alpha}\fint_{B_{5R}(\xi_1)} |\nabla u|^{\gamma} dz \le (5/4)^n 5^{-\alpha} \mathbf{M}_{\alpha}\left(|\nabla u|^{\gamma}\right)(\xi_1) \le 5^{n-\alpha} \lambda.
\end{align}
By this way, from~\eqref{cond:step2-b} and~\eqref{cond:step2-c}, it remains to control two last terms on the right-hand side of~\eqref{est:uk-inf-2} as follows 
\begin{align}\label{est:112}
R^{\alpha} \chi_2 \left[\fint_{B_{4R}(x)} |\nabla u|^{2-p}dz\right]^{\frac{\gamma}{2-p}} & \le R^{\alpha} \chi_2 \left[5^{n-\sigma} R^{-\sigma}  \mathbf{M}_{\sigma}\left(|\nabla u|^{2-p}\right)(\xi_2)\right]^{\frac{\gamma}{2-p}} \notag \\
& \le R^{\alpha} \left[5^{n-\sigma} R^{-\sigma} \left(\varepsilon^{-\gamma}\lambda\right)^{\frac{2-p}{\gamma}} \right]^{\frac{\gamma}{2-p}}  = 5^{\frac{\gamma(n-\sigma)}{2-p}} \varepsilon^{-\gamma} \lambda,
\end{align}
and
\begin{align}\label{est:113}
& R^{\alpha} \left[\frac{|\mu|(B_{4R}(x))}{R^{n-1}} + R\fint_{B_{4R}(x)} \left|\mathrm{div}\left(\mathbb{A}(\nabla \psi,z)\right)\right|dz\right]^{\frac{\gamma}{p-1}} \notag \\
& \hspace{2cm} \le R^{\alpha} \left[5^{n-\beta}\frac{R^{n-\beta}\mathbf{M}_{\beta}(\mu)(\xi_2)}{R^{n-1}} + 5^{n-\beta} R^{1-\beta} \mathbf{M}_{\beta}\left(\mathrm{div}\left(\mathbb{A}(\nabla \psi,\cdot)\right)\right)(\xi_2)\right]^{\frac{\gamma}{p-1}} \notag \\
& \hspace{2cm} \le R^{\alpha} \left[5^{n-\beta} R^{1-\beta} \left(\varepsilon^{2}\lambda\right)^{\frac{p-1}{\gamma}}\right]^{\frac{\gamma}{p-1}}  = 5^{\frac{\gamma(n-\beta)}{p-1}} \varepsilon^{2} \lambda.
\end{align}
It is worth mentioning here that $\alpha-\frac{\sigma\gamma}{2-p} = \alpha + \frac{\gamma(1-\beta)}{p-1} =0$ from~\eqref{beta-sigma}. Plugging~\eqref{est:111}, \eqref{est:112} and~\eqref{est:113} into~\eqref{est:uk-inf-2}, and choosing $\kappa \le \varepsilon^{\gamma}$, we infer that
\begin{align}\label{est:uk-inf-3}
\limsup_{k \to \infty} \mathbf{M}_{\alpha}^R\left(\chi_{B_{2R}(x)}|\nabla \tilde{u}_k|^{\gamma}\right)(y)  & \le C \left(1 + \kappa \chi_2 \varepsilon^{-\gamma} + \varepsilon^{2}\right) \lambda \le C_1 \lambda.
\end{align}
Analogously, from~\eqref{est:uk-vk} one has
\begin{align}
\limsup_{k \to \infty} \fint_{B_{2R}(x)} |\nabla u_k - \nabla \tilde{u}_k|^{\gamma} dz  & \le \kappa_1 \chi_1 \fint_{B_{2R}(x)} |\nabla u|^{\gamma} dz  + (\kappa_2+\delta)\chi_2 \left[\fint_{B_{2R}(x)} |\nabla u|^{2-p}dz\right]^{\frac{\gamma}{2-p}} \notag \\
&  + C \left[\frac{|\mu|(B_{2R}(x))}{R^{n-1}} + R\fint_{B_{2R}(x)} \left|\mathrm{div}\left(\mathbb{A}(\nabla \psi,z)\right)\right|dz\right]^{\frac{\gamma}{p-1}},\notag
\end{align}
which allows us to obtain
\begin{align}\label{est:uk-vk-2}
\limsup_{k \to \infty}\fint_{B_{2R}(x)} |\nabla u_k - \nabla \tilde{u}_k|^{\gamma} dz  & \le C_2 R^{-\alpha} \left(\kappa_1\chi_1  + (\kappa_2+\delta)\chi_2 \varepsilon^{-\gamma} + \varepsilon^2 \right)\lambda.
\end{align}
From~\eqref{est:uk-inf-3} and~\eqref{est:uk-vk-2}, there exists $k_0 \in \mathbb{N}$ such that 
\begin{align}\label{est:114}
\mathbf{M}_{\alpha}^R\left(\chi_{B_{2R}(x)}|\nabla \tilde{u}_k|^{\gamma}\right)(y) \le 2C_1 \lambda,
\end{align}
and 
\begin{align}\label{est:115}
\fint_{B_{2R}(x)} |\nabla u_k - \nabla \tilde{u}_k|^{\gamma} dz \le 2C_2 R^{-\alpha} \left(\kappa_1 \chi_1 + (\kappa_2+\delta)\chi_2 \varepsilon^{-\gamma} + \varepsilon^2 \right)\lambda,
\end{align}
hold for all $k \ge k_0$. The inequality~\eqref{est:114} comes out $\mathcal{N}_3 = 0$ for all $a/3 > 2C_1$ and $k \ge k_0$. We claim that from~\eqref{dec-N} and~\eqref{est:115} so that we deduce
\begin{align}
\mathcal{N} & \le  \left[\frac{C R^n}{a\lambda/3}\fint_{B_{2R}(x)} |\nabla u - \nabla u_k|^{\gamma} dz\right]^{\frac{n}{n-\alpha}} + \left[\frac{C R^n}{a\lambda/3}\fint_{B_{2R}(x)} |\nabla u_k - \nabla \tilde{u}_k|^{\gamma} dz\right]^{\frac{n}{n-\alpha}} \notag \\
& \le C_3 R^n \left[a^{-1} \left(\kappa_1\chi_1  + (\kappa_2+\delta)\chi_2 \varepsilon^{-\gamma} + \varepsilon^2 \right)\right]^{\frac{n}{n-\alpha}} + \left[\frac{CR^n}{a\lambda/3}\fint_{B_{2R}(x)} |\nabla u_k - \nabla \tilde{u}_k|^{\gamma} dz\right]^{\frac{n}{n-\alpha}},\notag
\end{align}
Send $k \to \infty$ to obtain
\begin{align}\label{est:116}
\mathcal{N} & \le C_4 \left[a^{-1} \left(\kappa_1\chi_1  + (\kappa_2+\delta)\chi_2 \varepsilon^{-\gamma} + \varepsilon^2 \right)\right]^{\frac{n}{n-\alpha}} \mathcal{L}^n \left(B_{R}(x)\right). 
\end{align}
For the purpose to prove~\eqref{eq:step2}, let us choose free parameters $\kappa_1, \kappa_2, \delta$ in \eqref{est:116} such that the exponent of $\varepsilon$ is greater than 1. In particular, we can take
$\kappa_1 = \varepsilon^2$ and $\kappa_2 = \delta = \varepsilon^{2+\gamma}$. Then, it allows us to find $\varepsilon_0 \in (0,1)$ such that
\begin{align*}
C_4 \left[ a^{-1} \left(\kappa_1 \chi_1  + (\kappa_2+\delta)\chi_2 \varepsilon_0^{-\gamma} + \varepsilon_0^2 \right)\right]^{\frac{n}{n-\alpha}} < \varepsilon_0.
\end{align*}

Next, we proceed the second case when $B_{16R}(x) \cap \partial \Omega \neq \emptyset$. With its aid,  one can find $y \in \partial \Omega$ such that $|y - x| = \mathrm{dist}(x,\partial \Omega) < 16R$. In this case, we will make use of Lemma~\ref{lem:comp3} for $B_{20R}(y) \supset B_{2R}(x)$. It means that there exists $\tilde{v}_k \in W^{1,p}(B_{40R}(y)) \cap W^{1,\infty}(B_{20R}(y))$ such that if $(\mathbb{A};\Omega) \in \mathcal{H}_{r_0,\delta}$ for $r_0>0$ and $\delta>0$, then for every $\kappa, \kappa_1, \kappa_2 \in (0,1)$, there holds
\begin{align}
\|\nabla \tilde{v}_k\|_{L^{\infty}(B_{20R}(y))}  & \le C \left(\fint_{B_{200R}(y)} |\nabla u_k|^{\gamma} dz\right)^{\frac{1}{\gamma}}  + \kappa \chi_2 \left[\fint_{B_{200R}(y)} |\nabla u_k|^{2-p}dz\right]^{\frac{1}{2-p}} \notag \\
&  + C \left[\frac{|\mu_k|(B_{200R}(y))}{R^{n-1}} + R\fint_{B_{200R}(y)} \left|\mathrm{div}\left(\mathbb{A}(\nabla \psi,z)\right)\right|dz\right]^{\frac{1}{p-1}},\notag
\end{align}
and
\begin{align}
 \fint_{B_{200R}(y)} |\nabla u_k - \nabla \tilde{v}_k|^{\gamma} dz  & \le \kappa_1\chi_1 \fint_{B_{200R}(y)} |\nabla u_k|^{\gamma} dz  + (\kappa_2+\delta)\chi_2 \left[\fint_{B_{200R}(y)} |\nabla u_k|^{2-p}dz\right]^{\frac{\gamma}{2-p}} \notag \\
&  + C \left[\frac{|\mu_k|(B_{200R}(y))}{R^{n-1}} + R\fint_{B_{200R}(y)} \left|\mathrm{div}\left(\mathbb{A}(\nabla \psi,z)\right)\right|dz\right]^{\frac{\gamma}{p-1}}.\notag
\end{align}
The remaining part of proof can be shown by the same argument as in the previous case. 
\end{proof}

\section{Proofs of main theorems}
\label{sec:main}

We now give detailed proofs of main Theorems by applying the following lemma. It is noteworthy that here, we make use of the Calder\'on-Zygmund (or Vitali) type of covering lemma, is known as Calder\'on-Zygmund-Krylov-Safonov decomposition, allowing to work with a family of balls instead of cubes (see \cite[Lemma 4.2]{CC1995}).

\begin{lemma}\label{lem:Vitali}
Let $\Omega$ be a $(r_0,\delta)$-Reifenberg flat domain, $0<  R_0 \le r_0$ and $D\subset E \subset \Omega$ be measurable sets. Suppose that
\begin{enumerate}
\item[i)] $\mathcal{L}^n(D) < \varepsilon \mathcal{L}^n(B_{R_0})$ for some $\varepsilon \in (0,1)$;
\item[ii)] for all $x\in \Omega$ and $\rho \in (0,R_0]$, if $\mathcal{L}^n(D\cap B_{\rho}(x)) \geq \varepsilon \mathcal{L}^n(B_{\rho}(x))$ then $B_{\rho}(x)\cap \Omega\subset E$.
\end{enumerate} 	
Then there is a constant $C = C(n)>0$ such that $\mathcal{L}^n(D) \leq C\varepsilon \mathcal{L}^n(E)$.
\end{lemma}

\begin{proof}[Proof of Theorem~\ref{theo:main-A}]
The level-set inequality~\eqref{dist-ineq} is a direct consequence of the following inequality 
$$\mathcal{L}^n \left(\mathbb{V}_1 \setminus \left(\mathbb{V}_2 \cup \mathbb{V}_3\right)\right) \le C \varepsilon \mathcal{L}^n \left(\mathbb{W}\right),$$ 
and with Lemma~\ref{lem:step1} in hand for two sets $D=\mathbb{V}$ and $E=\mathbb{W}$, respectively. Roughly speaking, we only need to show that $\mathbb{V}$ and $\mathbb{W}$ satisfying two hypotheses $i)$, $ii)$ of Lemma~\ref{lem:Vitali}. Certainly, the first one $i)$ is directly valid by Lemma~\ref{lem:step1}. To handle  the next claim $ii)$, we show by the contradiction. Indeed, let us assume that $B_R(x) \cap \Omega \not \subset \mathbb{W}$ for some $x \in \Omega$ and $0<R<r_0$. Then, taking Lemmas~\ref{lem:cutoff} and~\ref{lem:step2} into account, we shall reach a contradiction and the proof is complete. 
\end{proof}

\begin{proof}[Proof of Theorem~\ref{theo:main-B}]
Thanks to Theorem~\ref{theo:main-A}, there exist constants $a>0$, $\varepsilon_0 \in (0,1)$ and $\delta \in (0,1/2)$ such that if $(\mathbb{A};\Omega) \in \mathcal{H}_{r_0,\delta}$ for $r_0>0$, then~\eqref{dist-ineq} holds for any $\lambda>0$ and $\varepsilon \in (0,\varepsilon_0)$. Further, let $0<q<\infty$ and $0<s<\infty$, due to definition of Lorentz space $L^{q,s}(\Omega)$ and inequality~\eqref{dist-ineq} one has
\begin{align*}
\|\mathbf{M}_{\alpha}(|\nabla u|^{\gamma})\|_{L^{q,s}(\Omega)}^s & = a^s  q \int_0^\infty \lambda^{s-1}\mathcal{L}^n \left(\mathbb{V}_1\right)^{\frac{s}{q}} d\lambda \\
& \le C a^s  q \int_0^\infty \lambda^{s-1}\mathcal{L}^n \left(\mathbb{V}_2\right)^{\frac{s}{q}} d\lambda + C a^s q \int_0^\infty \lambda^{s-1}\mathcal{L}^n \left(\mathbb{V}_3\right)^{\frac{s}{q}} d\lambda \\
& \qquad + C a^s \varepsilon^{\frac{s}{q}} q \int_0^\infty \lambda^{s-1}\mathcal{L}^n \left(\mathbb{W}\right)^{\frac{s}{q}} d\lambda \\
& \le C a^s \varepsilon^{\gamma s} \chi_2 \left\|\left[\mathbf{M}_{\sigma}(|\nabla u|^{2-p})\right]^{\frac{\gamma}{2-p}}\right\|_{L^{q,s}(\Omega)}^s \\
& \qquad + C a^s \varepsilon^{-2s} \left\|\left[\mathbf{M}_{\beta}(\mu) + \mathbf{M}_{\beta}\left(\mathrm{div}\left(\mathbb{A}(\nabla \psi,\cdot)\right)\right)\right]^{\frac{\gamma}{p-1}}\right\|_{L^{q,s}(\Omega)}^s \\
& \qquad + C a^s \varepsilon^{\frac{s}{q}} \|\mathbf{M}_{\alpha}(|\nabla u|^{\gamma})\|_{L^{q,s}(\Omega)}^s,
\end{align*}
which leads to
\begin{align}\label{est:norm-a}
\|\mathbf{M}_{\alpha}(|\nabla u|^{\gamma})\|_{L^{q,s}(\Omega)} & \le C a \varepsilon^{\gamma} \chi_2 \left\|\left[\mathbf{M}_{\sigma}(|\nabla u|^{2-p})\right]^{\frac{\gamma}{2-p}}\right\|_{L^{q,s}(\Omega)} + C a \varepsilon^{\frac{1}{q}} \|\mathbf{M}_{\alpha}(|\nabla u|^{\gamma})\|_{L^{q,s}(\Omega)} \notag \\
& \qquad + C a \varepsilon^{-2} \left\|\left[\mathbf{M}_{\beta}(\mu) + \mathbf{M}_{\beta}\left(\mathrm{div}\left(\mathbb{A}(\nabla \psi,\cdot)\right)\right)\right]^{\frac{\gamma}{p-1}}\right\|_{L^{q,s}(\Omega)}.
\end{align}
Lately, let us set $\varepsilon_1 = \min\left\{\varepsilon_0, (2Ca)^{-1}\right\}$. From~\eqref{est:norm-a}, it is readily verified that
\begin{align*}
\|\mathbf{M}_{\alpha}(|\nabla u|^{\gamma})\|_{L^{q,s}(\Omega)} & \le 2C a \varepsilon^{\gamma} \chi_2 \left\|\left[\mathbf{M}_{\sigma}(|\nabla u|^{2-p})\right]^{\frac{\gamma}{2-p}}\right\|_{L^{q,s}(\Omega)} \\
&\qquad + 2 C a \varepsilon^{-2} \left\|\left[\mathbf{M}_{\beta}(\mu) + \mathbf{M}_{\beta}\left(\mathrm{div}\left(\mathbb{A}(\nabla \psi,\cdot)\right)\right)\right]^{\frac{\gamma}{p-1}}\right\|_{L^{q,s}(\Omega)},
\end{align*}
for every $\varepsilon \in (0,\varepsilon_1)$. On the other hand, when $s=\infty$, the latest inequality will be proved by the same argument. Hence, the desired result will be obtained by setting $\tilde{\epsilon} = 2C a \varepsilon_1^{\gamma}$.  
\end{proof}


\end{document}